\documentclass[12pt, reqno]{amsart}
\usepackage[utf8]{inputenc}
\usepackage[english]{babel}

\usepackage{fullpage}

\usepackage{ifthen}
\usepackage{microtype}
\usepackage[dvipsnames]{xcolor}
\usepackage{comment}
\usepackage{blkarray}
\usepackage{mathtools}
\usepackage{xfrac}
\usepackage{shuffle}
\usepackage{centernot}
\usepackage{enumitem}
\usepackage{mathdots}
\usepackage{multirow}
\usepackage{amssymb}
\usepackage{soul}
\usepackage{booktabs}
\usepackage{diagbox}
\usepackage{multicol}
\usepackage[noadjust]{cite}

\usepackage{amsmath}

\newlist{paraenum}{enumerate}{1}
\setlist[paraenum]{wide, label=(\arabic*)}

\makeatletter
\patchcmd{\@setaddresses}{\indent}{\noindent}{}{}
\patchcmd{\@setaddresses}{\indent}{\noindent}{}{}
\patchcmd{\@setaddresses}{\indent}{\noindent}{}{}
\patchcmd{\@setaddresses}{\indent}{\noindent}{}{}
\makeatother

\usepackage{caption}
\usepackage{subcaption}
\usepackage{tikz}    
\usetikzlibrary{cd}
\usetikzlibrary{backgrounds}
\usetikzlibrary{shapes}
\usetikzlibrary{arrows}
\usetikzlibrary{positioning}
\usetikzlibrary{calc}
\usetikzlibrary{quotes}
\usetikzlibrary{fit}
\usetikzlibrary{decorations.pathreplacing}

\usepackage[pdftex]{hyperref}
\hypersetup{
    colorlinks=true,
    linkcolor={magenta},
    citecolor={blue},
    urlcolor={blue}
}
\usepackage[capitalise]{cleveref}

\numberwithin{equation}{section}
\theoremstyle{plain}
\newtheorem{theorem}{Theorem}[section]
\newtheorem{conjecture}[theorem]{Conjecture}
\newtheorem{corollary}[theorem]{Corollary}
\newtheorem{lemma}[theorem]{Lemma}
\newtheorem{proposition}[theorem]{Proposition}

\theoremstyle{definition}
\newtheorem{remark}[theorem]{Remark}
\newtheorem{examplex}[theorem]{Example}
\newenvironment{example}
  {\pushQED{\qed}\examplex}
  {\popQED\endexamplex}

\newtheorem{definition}[theorem]{Definition}

\newcommand{\cc}{\mathbb{C}}

\newcommand{\zz}{\mathbb{Z}}

\newcommand{\kk}{\mathbb{K}}
\newcommand{\pp}{\mathbb{P}}

\newcommand{\cm}{\mathcal{M}}

\newcommand{\ci}{\mathcal{I}}
\newcommand{\cv}{\mathcal{V}}

\newcommand{\cb}{\mathcal{B}}

\DeclareMathOperator{\rank}{rank}

\DeclareMathOperator{\image}{image}

\DeclareMathOperator{\adj}{adj}
\DeclareMathOperator{\Adj}{Adj}

\newcommand{\monbar}{\overline{\mathcal{M}}_{0, n}}
\newcommand{\pt}{\mathrm{PT}}
\newcommand{\lc}{\mathrm{LC}}
\newcommand{\inv}{\mathrm{inv}}
\newcommand{\gr}{\mathrm{Gr}}
\DeclareRobustCommand{\abinom}{\genfrac{\langle}{\rangle}{0pt}{}}

\definecolor{benpurple}{RGB}{180, 0, 240}

\title{Parke--Taylor Varieties}

\author{Benjamin Hollering}
\address{Max Planck Institute for Mathematics in the Sciences, Inselstr. 22, 04103 Leipzig, Germany}
\email{benjamin.hollering@mis.mpg.de}

\author{Dmitrii Pavlov}
\address{Max Planck Institute for Physics, Boltzmannstr. 8, 85748 Garching bei M\"unchen, Germany}
\email{pavlov@mpp.mpg.de}


\date{}

\begin{document}

\begin{abstract}
Parke--Taylor functions are certain rational functions on the Grassmannian of lines encoding MHV amplitudes in particle physics. For $n$ particles there are $n!$ Parke--Taylor functions, corresponding to all orderings of the particles. Linear relations between these functions have been extensively studied in the last years. We here describe all non-linear polynomial relations between these functions in a simple combinatorial way and study the variety parametrized by them, called the Parke--Taylor variety. We show that the Parke--Taylor variety is equal to the log canonical embedding of the moduli space $\overline{\mathcal{M}}_{0,n}$ into $\mathbb{P}^{(n-2)!-1}$ due to Keel and Tevelev by exhibiting a concrete linear isomorphism with the realization in coordinates based on Kapranov's map. We also show that the intersection with the algebraic torus recovers the open part, $\mathcal{M}_{0,n}$. We give an explicit description of this isomorphism. Unlike the realization of the log canonical embedding by using the iterated Kapranov's map, this Parke--Taylor embedding very explicitly respects the symmetry of the $n$ marked points and is constructed in a single-step procedure, avoiding the intermediate embedding into a product of projective spaces.
\end{abstract}

\maketitle

\section{Introduction}
A Parke--Taylor function is a rational function on the (affine cone over) the Grassmannian $\mathrm{Gr}(2,n)$ of the form 
\begin{equation}\label{eq:ptf}
    f_\sigma = \dfrac{1}{\prod\limits_{i=1}^n p_{\sigma_i\sigma_{i+1}}},
\end{equation}
where $p_{ij}$ are the Pl\"ucker coordinates on $\mathrm{Gr}(2,n)$, $\sigma$ is a permutation on $[n]=\{1,\ldots,n\}$ and the indices are understood modulo $n$. In particle physics, Parke--Taylor functions are key components of the Parke--Taylor formula \cite{parke1986amplitude}, a strikingly simple expression for \emph{maximally helicity violating (MHV)} amplitudes. In this context $n$ is the number of particles in a scattering process. They also arise in numerous related contexts in scattering amplitudes, such as the CHY formalism \cite{cachazo2014scattering}, the momentum amplituhedron \cite{damgaard2021kleiss}, and CEGM amplitudes \cite{cachazo2024color} to name a few. 

For a scattering process of $n$ particles there are $n!$ Parke--Taylor functions, one for each ordering of the particles. Linear relations between these $n!$ functions, also called linear \emph{Parke--Taylor identities}, are also of interest in the aforementioned contexts, and have been studied in both the physics \cite{cachazo2019delta,AH2015nonplanar,frost2021lie,cachazo2024color,damgaard2021kleiss} and mathematics \cite{parisi2024magic} literature. However, to the best of our knowledge, there is no systematic description of \emph{polynomial} relations between the $n$-point Parke--Taylor functions. The main goal of this paper is to present such a description. 

More precisely, we are interested in the polynomial relations that are not implied by the linear Parke--Taylor identities. It is well-known that the $(n-2)!$ Parke--Taylor functions $f_\sigma$ with $\sigma_1=1$ and $\sigma_2=2$ form a basis of the $\mathbb{C}$-linear space spanned by all $n!$ Parke--Taylor functions. This is ensured by the Kleiss--Kuijf relations (see e.g. \cite[Lemma 2.2]{frost2021lie} or \cite{damgaard2021kleiss}). We will therefore only describe relations between such functions. The methods we use come from computational (nonlinear \cite{michalek2021invitation}) algebra. We define an \emph{algebraic variety} in $\mathbb{P}^{(n-2)!-1}$ \emph{parametrized} by the Parke--Taylor functions. That is, we consider the image of the map $\mathrm{Gr}(2,n)\dashrightarrow \mathbb{P}^{(n-2)!-1}$ given by these functions. We call this image the \emph{open Parke--Taylor variety} and denote it by $\mathrm{PT}^\circ_n$. Its Zariski closure is then called the \emph{Parke--Taylor variety} and is denoted by  $\mathrm{PT}_n$. We then \emph{implicitize} these varieties, that is, find their defining equations. 

While our study of Parke--Taylor varieties is primarily motivated by physics, they turn out to be quite interesting objects from a purely mathematical point of view. The ambient space of the Parke--Taylor variety $\mathrm{PT}_n$ is the projective space $\mathbb{P}^{(n-2)!-1}$. This is also the ambient space of the log canonical embedding of the moduli space of stable $n$-pointed rational curves $\overline{\mathcal{M}}_{0,n}$ constructed by Keel and Tevelev \cite{keel2009equations}. The moduli space $\overline{\mathcal{M}}_{0,n}$ is a compactification of $\mathcal{M}_{0,n}$, the moduli space of $n$ marked points on $\mathbb{P}^1$, known as the \emph{Deligne--Knudsen--Mumford} compactification.
The log canonical embedding is arguably the most standard way of realizing $\overline{\mathcal{M}}_{0,n}$ as a subvariety in projective space. 
In Theorem \ref{thm:log-canon-Parke--Taylor} we show that the Parke--Taylor variety $\mathrm{PT}_n$ is linearly isomorphic to the log canonical embedding $\mathrm{LC}_n$ of $\overline{\mathcal{M}}_{0,n}$, and that the linear isomorphism has an elegant, explicit combinatorial description.
We note that the construction of the log canonical embedding $\mathrm{LC}_n$ in coordinates involves iteratively applying Kapranov's map $\monbar\to\overline{\mathcal{M}}_{0,n-1}$ \cite{kapranov1993veronese} obtained by forgetting one marked point and stabilizing the resulting $n-1$-pointed curve. This breaks the symmetry of the $n$ marked points on a stable curve by forgetting them one at a time. Different reorderings of the points produce linearly isomorphic varieties but the isomorphisms are not obvious. 
Our construction is on the contrary almost completely symmetric: the only choice we make is that of the two fixed points $1$ and $2$. 
This can also be avoided by considering the variety in $\mathbb{P}^{n!-1}$ parametrized by all $n!$ Parke--Taylor functions, which makes the Parke--Taylor variety invariant under the action of $S_n$ on the indices of the Parke--Taylor functions.
Another important feature of this \emph{Parke--Taylor embedding} is that the open Parke--Taylor variety recovers the moduli space $\mathcal{M}_{0,n}$ (Proposition \ref{thm:PTM0n}). 

The task of describing the nonlinear Parke--Taylor identities is therefore the task of describing the equations of a specific projective embedding of $\monbar$. 
For various embeddings of this moduli space, this question has been extensively studied in the literature, and we now give a brief overview of the results.
For the log canonical embedding, Monin and Rana \cite{monin2017equations} conjectured a set of equations that define it scheme-theoretically, and verified this conjecture for $n\leq 8$.
This conjecture was proven for all $n$ by Gillespie, Griffin and Levinson in \cite{gillespie2022proof}. 
In \cite{maclagan2010chow}, Gibney and Maclagan embed $\monbar$ into a toric variety and describe the ideal of this embedding in the Cox ring of the toric variety up to saturation. 
As noted in \cite[Section 1]{monin2017equations}, this embedding is different from the log canonical embedding. 
Both these embeddings stem from the fact that $\monbar$ is a Chow quotient of the Grassmannian $\mathrm{Gr}(2,n)$ \cite{kapranov1993chow}. 
Finally, in \cite{vakil2009equations} the authors study a rather general question of determining the equations for the moduli space of $n$ points on a line in the presence of weights, and their starting point is representing the moduli space as a GIT quotient of a product of projective lines.

The explicit linear isomorphism between the Parke--Taylor variety $\mathrm{PT}_n$ and the log canonical embedding $\mathrm{LC}_n$ of $\monbar$, in combination with the results of \cite{monin2017equations, gillespie2022proof}, suggests one way of describing the defining equations of $\mathrm{PT}_n$. One can just take the equations of $\mathrm{LC}_n$ and apply the linear isomorphism from Theorem \ref{thm:log-canon-Parke--Taylor} to them. This description, however, is not particularly explicit. The log canonical embedding is the composition of an embedding $\monbar \hookrightarrow \mathbb{P}^1\times\ldots\times\mathbb{P}^{n-3}$ and the Segre embedding $\mathbb{P}^1\times\ldots\times\mathbb{P}^{n-3}\hookrightarrow \mathbb{P}^{(n-2)!-1}$, and \cite{monin2017equations, gillespie2022proof} concentrate on describing the equations of $\monbar$ in the product of projective spaces, before the Segre embedding. The work of Keel and Tevelev \cite{keel2009equations} provides a way to transform these into equations in $\mathbb{P}^{(n-2)!-1}$, which is compactly described in \cite[Section 1]{monin2017equations}. However, this transformation is not entirely trivial: for instance, the equations of Monin and Rana are cubic; however, the ideal in $\mathbb{P}^{(n-2)!-1}$ is quadratically generated \cite[Theorem 8.6]{keel2009equations}. For this reason, we opt for presenting a more explicit description of the ideal of $\mathrm{PT}_n$. Here we use the fact that the Parke--Taylor embedding is defined in a more straightforward way than the coordinate realization of the log canonical embedding from \cite{keel2009equations, monin2017equations} (and the embeddings of \cite{maclagan2010chow} and \cite{vakil2009equations}): it is the closure of the image of a single rational (essentially, monomial) map from $\mathrm{Gr}(2,n)$ directly into $\mathbb{P}^{(n-2)!-1}$.

Our description gives the ideal of $\mathrm{PT}_n$ up to saturation by the product of the variables. It is based on the following observation: the map $\varphi_n:\mathrm{Gr}(2,n)\dashrightarrow \mathbb{P}^{(n-2)!-1}$ defining $\mathrm{PT}_n$ is a monomial map in the inverses $\frac{1}{p_{ij}}$ of the Pl\"ucker variables. This map can be extended to the whole Pl\"ucker space $\mathbb{P}^{\binom{n}{2}-1}$. The closure of the image of this extended map is a toric variety, which we denote by $T_n$, and its defining equations are binomial. We then have a natural inclusion $\mathrm{PT}_n\subset T_n$. That is, similarly to \cite{maclagan2010chow}, we embed $\overline{\cm}_{0,n}$ into a toric variety. Our toric variety, however, is different from that in \cite{maclagan2010chow}, see Section \ref{sec:toric} for details. The equations of $\mathrm{PT}_n$ therefore come in two flavors: those of the first type are the binomial equations of $T_n$, and those of the second type are \emph{lifts} of the Pl\"ucker relations defining $\mathrm{Gr}(2,n)$ inside $\mathbb{P}^{\binom{n}{2}-1}$, again resembling the setup of \cite{maclagan2010chow}. We describe these two types of equations in Sections \ref{sec:toric} and \ref{sec:idealPT} respectively.

Finally, we note that this topic also has an algebraic statistics flavor. The real positive points of the Parke--Taylor variety $\mathrm{PT}_n$ can be regarded as a statistical model for ranked data as considered in \cite{sturmfels2012statistical}. We note that such a \emph{Parke--Taylor model} is distinct from the three toric models studied in \cite{sturmfels2012statistical} for all $n \geq 5$. In particular, the toric Parke--Taylor model is lower-dimensional than the other models and thus its ideal is generated by more polynomials, which we conjecture have unbounded degree as $n$ grows. 

The remainder of this article is organized as follows. In Section \ref{sec:prelim} we give some basic notions and define Parke--Taylor varieties.
Although the notions we describe are quite basic from the point of view of an algebraist, we find it useful to present them here to make the article accessible to a broad range of readers, including those from the physics community.
In Section \ref{sec:PTM0n} we connect Parke--Taylor varieties to moduli spaces. Proposition \ref{thm:PTM0n} establishes an isomorphism between the open Parke--Taylor variety $\mathrm{PT}_n^\circ$ and $\mathcal{M}_{0,n}$, and Theorem \ref{thm:log-canon-Parke--Taylor} gives a linear isomorphism between $\mathrm{PT}_n$ and the log canonical embedding $\mathrm{LC}_n$ of $\monbar$. We use this to establish the degree of $\mathrm{PT}_n$ and show that its ideal is generated by quadrics.
In Section \ref{sec:toric} we take a first step towards describing the equations that cut out $\mathrm{PT}_n$. 
More precisely, we describe all \emph{binomial} relations that hold on this variety. 
They cut out the \emph{toric Parke--Taylor variety} $T_n$ which contains $\mathrm{PT}_n$. 
These binomial relations admit a particularly simple combinatorial description, which we present in Proposition \ref{prop:toric-gens-via-adjs}. 
In Theorem \ref{thm:toricquadratics} we describe a set of \emph{quadratic} binomials, which we conjecture is enough to generate the ideal of $T_n$ for $n\geq 7$ up to saturation by the product of the variables. 
We verify this for $n\leq 9$.
Finally, in Section \ref{sec:idealPT} we describe the full ideal of the Parke--Taylor variety. 
In Theorem \ref{thm:fullideal} we show that to generate the ideal of $\mathrm{PT}_n$ it is enough to know all the binomial relations (that is, the ideal of $T_n$) and the \emph{lift} of $\mathrm{Gr}(2,n)$.
We carefully describe the full ideal of the Parke--Taylor variety $\mathrm{PT}_6$ in Example \ref{example:fullideal6}.

\section{Preliminaries} \label{sec:prelim}

The goal of this paper is to describe all algebraic relations between the $n$-point Parke--Taylor functions. In this section we formulate this goal in the language of varieties and ideals. For more details on these objects, see \cite{cox1997ideals}. In what follows we write $\mathbb{C}^n$ and $\mathbb{P}^n$ for the complex $n$-dimensional affine and projective space respectively. The ring of polynomials in the variables $x_1,\ldots,x_n$ is denoted by $\mathbb{C}[x_1,\ldots,x_n]$.

Most of the varieties we study in this paper will be \emph{parameterized}, meaning they are given as the closure of the image of a rational map $\varphi: \pp^m \dashrightarrow \pp^n$. Every such polynomial map of projective (or affine) spaces induces a corresponding map of polynomial rings, $\varphi^\ast: \cc[x_1, \ldots, x_n] \to \cc[t_1, \ldots, t_m]$ which is given by $\varphi^\ast(x_i) = \varphi_i$ where $\varphi_i$ is the $i$-th coordinate function of the rational map $\varphi$. Moreover, it holds that $\ker(\varphi^\ast) = \ci(\image(\varphi))$. In many cases, the maps we study will be rational and thus we will need to account for the fact that our map is not defined on the locus given by the denominators. One tool we will use to deal with this problem is the following. 

\begin{definition}
    Let $R$ be a commutative ring, $I\subset R$ be an ideal and $a\in R$. The \emph{saturation} of $I$ by $a$ is
    $$I:a^\infty := \{r\in R\ |\ \exists n\in\mathbb{Z}_{\geq 0}\ :\ ra^n\in I \}.$$
    Geometrically, when $R$ is a polynomial ring, saturation corresponds to removing the irreducible components of $\mathcal{V}(I)$ contained in the variety defined by $a$.
\end{definition}

Throughout the paper we will sometimes use more advanced notions from algebraic geometry and combinatorial commutative algebra. However, since these will only appear in proofs or remarks, we refrain from defining all of them here.  

We now define two classical varieties that will play a role in our story. 
The first one is the \emph{Grassmannian} $\mathrm{Gr}(k,n)$, the algebraic variety parametrizing $k$-dimensional linear subspaces of $\mathbb{C}^n$. It admits an embedding into $\mathbb{P}^{\binom{n}{k}-1}$ by taking the $k\times k$ minors of a $k\times n$ matrix whose rows span the linear subspace. These minors are called \emph{Pl\"ucker coordinates} and the defining equations of this embedding are called \emph{Pl\"ucker relations}. For more details on the Grassmannian, see e.g. \cite[Chapter 5]{michalek2021invitation}. The \emph{open Grassmannian} $\mathrm{Gr}^\circ(k,n)$ is the subset on which all Pl\"ucker coordinates are non-zero. For the purposes of this article we will be interested in the Grassmannian $\mathrm{Gr}(2,n)$. For this Grassmannian, the Pl\"ucker relations are particularly simple: they have the form $p_{ij}p_{kl}-p_{ik}p_{jl}+p_{il}p_{jk}=0$ for all choices of indices such that $1\leq i < j<k<l\leq n$. 

The second variety of interest is $\mathcal{M}_{0,n}$, the moduli space of $n$ (distinct) marked points on $\mathbb{P}^1$. The configurations of $n$ points are considered up to projective transformations, which allows to fix three of the points to be $0,1$ and $\infty$. This allows to identify $\mathcal{M}_{0,n}$ with the subset of $\mathrm{Gr}(2,n)$ given by rowspans of matrices of the form 
$$
\begin{pmatrix}
    1 & 1 & 1 & \ldots & 1 & 0\\
    0 & 1 & x_1 & \ldots & x_{n-3} & 1
\end{pmatrix}.
$$
In fact, $\mathcal{M}_{0,n}$ is isomorphic to the quotient of $\mathrm{Gr}^\circ(2,n)$ by the column scaling action of the algebraic torus $(\mathbb{C}^*)^n$. That is, $\mathcal{M}_{0,n}\cong \mathrm{Gr}^\circ(2,n)/(\mathbb{C}^*)^n$. We have that $\mathcal{M}_{0,n}$ is an irreducible variety of dimension $n-3$. We note that for $n\geq 3$, $\mathcal{M}_{0,n}$ is a \emph{very affine} variety, i.e. a closed subvariety of the algebraic torus. As such, for $n\geq 4$ it is not compact. A natural compactification is the moduli space of stable $n$-pointed rational curves $\monbar$, defined by Deligne--Mumford and Knudsen, which is a projective variety. More details on $\mathcal{M}_{0,n}$ and its role in particle physics are available in the introductory reference \cite{lam2024moduli}. 

We will now define Parke--Taylor functions as certain rational functions on (the affine cone over) $\mathrm{Gr}(2,n)$. More precisely, let $S_n$  be the symmetric group on $[n]$. For every element $\sigma \in S_n$ we define the Parke--Taylor function $f_\sigma$ by Equation \eqref{eq:ptf}. The indices in this equation are understood modulo $n$, i.e. $\sigma_{n+1}:=\sigma_1$. 

We are interested in describing all polynomial relations between the Parke--Taylor functions $f_\sigma$ for a fixed $n$. Linear relations between these functions are relatively well-understood. Firstly, the Parke--Taylor functions are invariant under cyclic shifts of indices, i.e. $f_{(\sigma_1\ldots \sigma_n)}=f_{(\sigma_2\ldots \sigma_n \sigma_1)}$. This allows to fix $\sigma_1=1$. Secondly, flipping the remaining unfixed part of the permutation also does not change the Parke--Taylor function, i.e. $f_{(1,\sigma_2,\ldots,\sigma_n)}=f_{(1,\sigma_n,\ldots,\sigma_2)}$. Both these relations are straightforward to verify and hold on the affine cone $\mathbb{C}^{\binom{n}{2}}$ over the ambient projective space of $\mathrm{Gr}(2,n)$.
Finally, on $\mathrm{Gr}(2,n)$ the Parke--Taylor functions satisfy the so-called \emph{Kleiss--Kuijf relations} \cite{kleiss1989multigluon}. For any permutation of the form $(1\alpha n\beta)$, where $\alpha \cup \beta = \{2,\ldots,n-1\}$, the corresponding Kleiss-Kuijf relation is 
$$f_{(1\alpha n\beta)} = (-1)^{|\beta|}\sum\limits_{\sigma \in \alpha\shuffle \beta^T}f_{(1\sigma n)},$$
where $|\beta|$ is the cardinality of $\beta$, $\beta^T$ is the reverse ordering of $\beta$ and $\alpha\shuffle \beta^T$ is the set of all permutations on $\{2,\ldots,n\}$ such that the elements from $\alpha$ and $\beta^T$ appear in order given by these permutations. See \cite[Sec. 7]{parisi2024magic} for a rigorous mathematical proof. 
In particular, as explained in the Introduction, these relations ensure that a basis of the $\mathbb{C}$-linear space spanned by the Parke--Taylor functions is given by the functions $f_{\sigma}$ with $\sigma_1=1$ and $\sigma_2=2$. In order to reduce the computational complexity of our problem, we will only consider polynomial relations between such Parke--Taylor functions. These can be thought of as ``non-linear relations that do not follow from the linear ones''. 

\begin{definition}\label{def:pt}
    Let $\Sigma_n$ denote the subgroup of $S_n$ consisting of the permutations that fix the first two elements. Note that $\Sigma_n\cong S_{n-2}$ and in particular $|\Sigma_n|=(n-2)!$. We now consider the $((n-2)!-1)$-dimensional projective space $\mathbb{P}^{(n-2)!-1}$ whose coordinates $z_\sigma$ are labeled by the elements of $\Sigma_n$. We then define the rational map 
$\varphi_n: \mathrm{Gr}(2,n)\dashrightarrow \mathbb{P}^{(n-2)!-1}$ by $V\mapsto z$, where $z_\sigma(V) = f_\sigma(V)$. The image $\varphi_n(\mathrm{Gr}(2,n))$ is called the \emph{open Parke--Taylor variety} $\mathrm{PT}^\circ_n$. Its Zariski closure in $\mathbb{P}^{(n-2)!-1}$ is called the \emph{Parke--Taylor variety} $\mathrm{PT}_n$.
\end{definition}
In the language of commutative algebra, describing all polynomial relations that hold between the Parke--Taylor functions means describing the ideal of $\mathrm{PT}_n$.

\subsection{The Weak Order on $S_n$}
In Section \ref{sec:PTM0n}, we will make some connections between the Parke--Taylor variety $\pt_n$ and the well-known \emph{weak order} which is a partial order on the symmetric group $S_n$. We give a brief description of this partial order on permutations here and refer the reader to \cite{bjornerbrenticoxeter} for more information. We note that the description we give here is tailored specifically to the weak order on the symmetric group, instead of a general Coxeter group, since that will be our focus in this paper.

Let $\sigma \in S_n$ be a permutation and denote the \emph{value-wise} inversion set of $\sigma$ by $\inv(\sigma) = \{(\sigma_i, \sigma_j) ~:~  i < j \text{ and } \sigma_i > \sigma_j \}$. That is, a value-wise inversion is a pair of numbers that appear out of order in $\sigma$
when written in one-line notation.

\begin{definition}
Let $S_n$ be the symmetric group on $n$ elements. 
Then the \emph{right weak order} is the poset $(S_n, \leq)$ defined by $\sigma \leq \tau$ if and only if $\inv(\sigma) \subseteq \inv(\tau)$.
\end{definition}

It is well known that the weak order can be defined for any Coxeter group and that for finite Coxeter groups, it is also an order-theoretic lattice \cite{bjornerbrenticoxeter}. The weak order has many deep connections to various other combinatorial objects, such as its appearance as the 1-skeleton of the \emph{permutohedron}, which is pictured in Figure \ref{fig:weak-order}. 

\begin{definition}
A \emph{lower order ideal} in a poset $(P, \leq)$ is a set $I \subset P$ such that if $s \in P$ and $t \leq s$ then $t \in P$. A set $S = \{s_1, \ldots, s_k\} \subset P$ generates the lower order ideal if $I = \{t \in P ~:~ t \leq s \text{ for some } s \in S\}$ and we write $I = \langle s_1, \ldots, s_k \rangle$.  A lower order ideal $I$ is \emph{principal} if $I = \langle s \rangle$. 
\end{definition}

\begin{figure}
\centering
\begin{tikzpicture}[scale = .5, thick]
    \draw (3,0)--(5,2);
    \draw (3,0)--(1,2);
    \draw (5,3)--(5,5);
    \draw (1,3)--(1,5);
    \draw (1,6)--(3,8);
    \draw (5,6)--(3,8);
    \draw (3,0) node[below]{$123$};
    \draw (1,2) node[above] {$213$};
    \draw (5,2) node[above] {$132$};
    \draw (1,5) node[above] {$312$};
    \draw (5,5) node[above] {$231$};
    \draw (3,8) node[above] {$321$};
\end{tikzpicture}
~~~~~
\begin{tikzpicture}[scale = .5, thick]
    \draw (3,0)--(5,2);
    \draw (3,0)--(1,2);
    \draw (5,3)--(5,5);
    \draw (1,3)--(1,5);
    \draw (1,6)--(3,8);
    \draw (5,6)--(3,8);
    \draw (3,0) node[below]{$123$};
    \draw (1,2) node[above] {$213$};
    \draw (5,2) node[above] {$132$};
    \draw (1,5) node[above] {$231$};
    \draw (5,5) node[above] {$312$};
    \draw (3,8) node[above] {$321$};
\end{tikzpicture}
\caption{The left and right weak order on $S_3$ pictured on the left and right respectively.}
\label{fig:weak-order}
\end{figure}
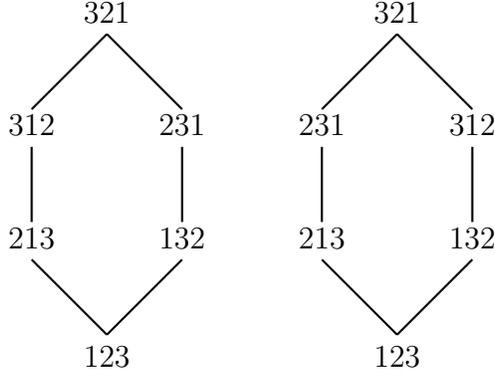

\section{Parke--Taylor varieties and $\monbar$} \label{sec:PTM0n}
In this section we describe some geometric properties of the Parke--Taylor variety $\mathrm{PT}_n$. In particular, we show that the open Parke--Taylor variety $\pt_n^\circ$ is isomorphic to the moduli space of $n$ marked points on the projective line, $\cm_{0,n}$. Moreover, the Parke--Taylor variety $\pt_n$ is isomorphic to the Deligne--Knudsen--Mumford compactification, $\monbar$ and thus gives a new embedding of $\monbar$ into $\pp^{(n-2)! - 1}$ which is linearly isomorphic to the well known embedding of Keel and Tevelev \cite{keel2009equations} realized by composing the iterated version of Karpanov's map \cite{kapranov1993veronese} with the Segre embedding of $\mathbb{P}^\times\ldots\times \mathbb{P}^{n-3}$.
We then exploit this relationship between $\pt_n$ and $\monbar$ to obtain the degree and dimension of $\pt_n$. In Section \ref{sec:idealPT} we will further elaborate on how this structure also allows us to better understand the generators of the vanishing ideal $\ci(\pt_n)$, and thus determine polynomial relations which hold for Parke--Taylor functions. 

\begin{proposition} \label{thm:PTM0n}
    The open Parke--Taylor variety $\mathrm{PT}^\circ_n$ is isomorphic to $\mathcal{M}_{0,n}$. 
\end{proposition}

\begin{proof}
    Note that the map $\varphi_n: \mathrm{Gr}(2,n)\dashrightarrow \mathrm{PT}_n$ is not birational, its fibers have positive dimension.
    This stems from the fact that scaling the columns of a matrix representative of $V\in\mathrm{Gr}(2,n)$ does not change the point $\varphi_n(V)$ in $\mathrm{PT}_n$. Moreover, it is only defined on the open Grassmannian $\mathrm{Gr}^\circ(2,n)$, i.e. on the locus where all Pl\"ucker coordinates are non-zero.
    This means that $\varphi_n$ induces a surjective regular map $\mu_n: \mathrm{Gr}^\circ(2,n)/(\mathbb{C}^*)^n\to \mathrm{PT}^\circ_n\subset \mathbb{P}^{(n-2)!-1}$. The domain $\mathrm{Gr}^\circ(2,n)/(\mathbb{C}^*)^n$ is the moduli space $\mathcal{M}_{0,n}$. The varieties $\mathrm{PT}_n^\circ$ and $\cm_{0,n}$ are in fact isomorphic.
    To show this, for any quadruple of indices $(i,j,k,l)$, consider two Parke--Taylor functions that correspond to permutations that only differ by the transposition $(ij)$ and such that the first one involves the product $p_{ki}p_{ij}p_{jl}$ and the second one the product $p_{kj}p_{ij}p_{il}$, with all the other variables in both functions being the same. Then the ratio of these two Parke--Taylor functions is the cross-ratio $[ij|kl]=\dfrac{p_{ik}p_{jl}}{p_{jk}p_{il}}$. In case one of the Parke--Taylor functions in this ratio corresponds to a permutation that does not fix $1$ and $2$, we express it a linear combination of the Parke--Taylor functions that have this property. 
    The cross-ratios $[ij|kl]$ define a unique point of $\mathcal{M}_{0,n}$ \cite[Proposition 1.6]{lam2024moduli}, which shows that $\mathrm{PT}^\circ_n$ maps injectively to $\mathcal{M}_{0,n}$. Together with the surjectivity of $\mu_n$ this yields the claim. 
\end{proof}

\begin{remark}
\label{rem:restrict-to-m0n}
One consequence of the above theorem is that we may restrict the domain of the parametrization $\varphi_n$ to $\cm_{0, n}$ without changing the image. Recall that each element of $\cm_{0, n}$ can be represented by a matrix of the form:
\[
\begin{pmatrix}
    1 & 1 & 1 & \cdots & 1 & 1\\
    p_1 & p_2 & p_3 & \cdots &p_{n-1} &p_n
\end{pmatrix},
\]
where one puts $p_1=0,\ p_2=1,\ p_n=\infty$ (this last equality can also be interpreted as replacing the last column with $(0\ 1)^T$) and $p_i\neq p_j$.
The Pl\"ucker embedding of a matrix of this form is then of course given by $p_{ij} = p_j - p_i$ and so $\varphi_n$ can be rewritten as
\begin{align*}
\varphi_n: \cm_{0, n} &\to \pp^{(n-2)! - 1} \\
            (p_1, \ldots p_n) &\mapsto \left(  \prod_{i=1}^n \dfrac{1}{p_{\sigma_i} - p_{\sigma_{i+1}}}  ~:~ \sigma \in \Sigma_n \right).
\end{align*}
By convention, we have $(p_i-p_{n})^{-1}=1$ for $i<n$. This is because the $(in)$ minor of the matrix above is equal to one when $(1\ p_n)$ is replaced by $(0\ 1)$.
Throughout the remainder of this paper we will abuse notation and denote by $\varphi_n$ the map which appears in Definition \ref{def:pt} and that which is above, as the images of these maps coincide. We will also sometimes use $p_{ij}:=p_i-p_j$: since our varieties are defined by homogeneous functions in the Pl\"ucker variables, this global change of sign will not change them. 
\end{remark}

Since $\mathrm{PT}_n^\circ$ is isomorphic to $\cm_{0, n}$, a natural next question is to determine if the closure $\pt_n$ is isomorphic to a known compactification of $\cm_{0, n}$, e.g., to $\monbar$. This is indeed the case and thus the parametrization of the Parke--Taylor variety in Definition \ref{def:pt} gives an embedding of $\monbar$. In fact, it realizes the \emph{log canonical embedding} of $\overline{\cm}_{0,n}$ \cite{keel2009equations}. We show this by exhibiting a linear isomorphism between $\pt_n$ and a particular coordinate realization of the log canonical embedding described in Definition \ref{def:logcancoord}. This is the content of Theorem \ref{thm:log-canon-Parke--Taylor}. We begin with a brief overview of the log canonical embedding of $\monbar$.

Recall that the moduli space $\monbar$ parametrizes \emph{stable $n$-pointed rational curves} which are represented as tuples $(C, p_1, \ldots, p_n)$. In \cite{keel2009equations}, the authors introduced a map $\Omega_n$, called the \emph{iterated Kapranov embedding} which embeds $\monbar$ into the product of projective spaces $\pp^1 \times \pp^2 \times \cdots \times \pp^{n-3}$. This map naturally arises from repeatedly applying the Kapranov morphism $\Psi_n: \overline{\cm_{0, n}} \to \pp^{n-3}$ which was introduced in \cite{kapranov1993veronese} and the forgetful map $\pi_n: \monbar \to \overline{\cm_{0, n-1}}$ which is given by forgetting the last point $p_n$ and stabilizing the curve $C$. The following theorem of \cite{keel2009equations} will allow us to describe $\Omega_n$ explicitly in coordinates.

\begin{theorem}\cite[Corollary 2.7]{keel2009equations}
The map $\Omega_n: \monbar \to \pp^1 \times \pp^2 \times \cdots \times \pp^{n-3}$ is a closed embedding. 
\end{theorem}

As noted in \cite{monin2017equations}, since $\cm_{0, n}$ is Zariski dense in $\monbar$, and $\Omega_n$ is a closed embedding, it suffices to describe $\Omega_n$ on $\cm_{0, n}$. We now define the log canonical embedding of $\overline{\cm}_{0,n}$ as in \cite{keel2009equations}. 
\begin{definition}
   Let $B$ be the boundary divisor of $\overline{\cm}_{0,n}$ and $K_{\overline{\cm}_{0,n}}$ be the canonical divisor. Then the divisor $D:=K_{\overline{\cm}_{0,n}}+B$ is very ample \cite[Theorem 1.1]{keel2009equations} and defines an embedding $\overline{\cm}_{0,n}\to \mathbb{P}(H^0(\overline{\cm}_{0,n}, \mathcal{O}(D))^*)\cong \mathbb{P}^{(n-2)!-1}$. This is the \emph{log canonical embedding} of $\overline{\cm}_{0,n}$.
\end{definition}

Throughout the remainder of this section we will use the following explicit description of the log canonical embedding from \cite{monin2017equations}. This description is originally due to Kapranov \cite{kapranov1993chow,kapranov1993veronese}.

\begin{definition} \label{def:logcancoord}
The \emph{log canonical embedding} $\Phi_n$ of $\monbar$ is the composition of the Segre embedding $\xi$ of $\pp^1 \times \pp^2 \times \cdots \times \pp^{n-3}$ with the \emph{iterated Kapranov embedding} $\Omega_n$ which is
\begin{align*}
\Omega_n: \monbar &\to \pp^1 \times \pp^2 \times \cdots \times \pp^{n-3} 
\end{align*}
where for a tuple $(C, p_1, \ldots, p_n) \in \cm_{0, n}$ (with $p_1=0,\ p_2=1,\ p_n=\infty$) the $i$-th component is given by $(\Omega_n(C, p_1, \ldots, p_n))_i = [\frac{p_1 -p_2}{p_{i+3}-p_2}:\cdots:\frac{p_1 -p_{i+2}}{p_{i+3}-p_{i+2}}]$. We denote the closure of the image of $\Phi_n$ by $\lc_n = \overline{\mathrm{image}(\Phi_n)}$. We denote the coordinates of the image by $t_{i_1 \ldots i_{n-3}}$.
\end{definition}

\begin{example}
\label{ex:log-canon-5}
Consider the case $n = 5$ in which the log canonical embedding $\Phi_n$ is given by the composition $\Phi_5 = \xi  \circ  \Omega_5$ where $\xi$ is the Segre embedding of $\pp^1 \times \pp^2$ and $\Omega_5: \overline{\cm_{0, 5}} \to \pp^1 \times \pp^2$ is the iterated Kapranov embedding which is given by
\[
\Omega_5(p_1 \ldots, p_5) = \left[\frac{p_1 - p_2}{p_4 - p_2} : \frac{p_1 - p_3}{p_4 -p_3}\right], \left[\frac{p_1 - p_2}{p_5 - p_2} : \frac{p_1 - p_3}{p_5 -p_3} : \frac{p_1 - p_4}{p_5 -p_4}\right]
\]
for 5 marked points $(p_1, \ldots, p_5) \in \cm_{0, n}$. Applying the Segre embedding to the image of $\Omega_5$, yields the log canonical embedding which is then given by
\begin{align*}
t_{11} = \frac{(p_1 - p_2)^2}{(p_4 - p_2)(p_5 - p_2)},~~  t_{12} = \frac{(p_1 - p_2)(p_1 - p_3)}{(p_4 - p_2)(p_5 - p_3)},~~  t_{13} = \frac{(p_1 - p_2)(p_1 - p_4)}{(p_4 - p_2)(p_5 - p_4)},\\
t_{21} = \frac{(p_1 - p_3)(p_1 - p_2)}{(p_4 - p_3)(p_5 - p_2)},~~  t_{22} = \frac{(p_1 - p_3)^2}{(p_4 - p_3)(p_5 - p_3)},~~ t_{23} = \frac{(p_1 - p_3)(p_1 - p_4)}{(p_4 - p_3)(p_5 - p_4)}.
\end{align*}

One can verify with direct computation in a computer algebra system such as \texttt{Macaulay2} \cite{M2}, that the vanishing ideal of $\lc_5 = \mathrm{image}(\Phi_5)$ is generated by the 5 quadratics which appeared in \cite[Corollary 1.3]{monin2017equations}:
\begin{align*}
\ci(\lc_5) =  
\langle
t_{11}t_{12}-t_{12}t_{13}-t_{11}t_{22}+t_{11}t_{23}, t_{12}t_{21}-t_{21}t_{22}-t_{12}t_{23}+t_{21}t_{23}\\
-t_{11}t_{22}+t_{13}t_{22}+t_{21}t_{22}-t_{21}t_{23}, -t_{13}t_{21}+t_{11}t_{23}, t_{13}t_{22}-t_{12}t_{23}
\rangle. 
\end{align*}
Moreover, since this ideal is codimension 3 and is Gorenstein, there exists an antisymmetric matrix such that the ideal is generated by its $4 \times 4$ sub-Pfaffians \cite[Theorem 2.1]{buchsbaum1977algebra}. For $\ci(\lc_5)$, the following matrix realizes this result: 
\begin{align*}
M_\lc = \begin{pmatrix}
0& 0& t_{22}-t_{23}& -t_{11}+t_{13}& -t_{21}+t_{23} \\
0& 0&-t_{12}+t_{22}-t_{23}& t_{13}& t_{23} \\
-t_{22}+t_{23}& t_{12}-t_{22}+t_{23}& 0& -t_{22}+t_{23}& -t_{22} \\
t_{11}-t_{13}&-t_{13}& t_{22}-t_{23}& 0& t_{23} \\
t_{21}-t_{23}& -t_{23}& t_{22}&-t_{23}& 0
\end{pmatrix}.
\end{align*}
\end{example}

\begin{remark}
    The image of the map $\Omega_n$ has been extensively studied in the literature. In the original paper \cite{keel2009equations} conjecture a set of equations in the Cox ring of $\mathbb{P}^1\times\ldots\times \mathbb{P}^{n-3}$ that define the image as a scheme. These equations were made more explicit in \cite{monin2017equations} and the conjecture was proven for $n\leq 8$. The general case was settled in \cite{gillespie2022proof}. The multidegree of this image was determined in \cite{cavalieri2021projective} in terms of the asymmetric multinomial coefficients. Moreover, by \cite[Theorem 8.6]{keel2009equations} the ideal of the log canonical embedding embedding in $\mathbb{P}^{(n-2)!-1}$ is quadratically generated. 
\end{remark}

To state the main theorem of this section we will need the following definition which we use in the construction of our linear map between $\lc_n$ and~$\pt_n$.

\begin{definition}
\label{defn:shuffle-sum}
Let $\inv(\sigma)$ denote the set of value-wise inversions set of $\sigma$, which is given by $\inv(\sigma) = \{(\sigma_i, \sigma_j) ~:~  i < j \text{ and } \sigma_i > \sigma_j \}$. Then for any $i \in [n-3]$ and $\sigma \in \Sigma_{n-1}$, let $\shuffle_i z_\sigma \in \cc[z_\tau ~:~ \tau \in \Sigma_n
]$ be the linear polynomial
\[
\displaystyle \shuffle_{i} z_\sigma = \sum_{\substack{\tau \in \Sigma_n \\ \tau \setminus \{n \} = \sigma \\ (i, n) \notin \inv(\tau)}} z_\tau. 
\]
Observe that the sum here is simply over $\tau$ obtained from $\sigma$ by inserting $n$ anywhere to the right of $i$. 
Note that $\shuffle_i$ is different from the shuffle product $\shuffle$ in the definition of the Kleiss--Kuijf relations in Section \ref{sec:prelim}. 
\end{definition}

\begin{theorem}
\label{thm:log-canon-Parke--Taylor}
The Parke--Taylor variety $\pt_n$ realizes the log canonical embedding $\lc_n$ and thus is a closed embedding of $\monbar$ into $\pp^{(n-2)! - 1}$.
For $n\ge6$ the linear isomorphism with the realization in Definition \ref{def:logcancoord} is described recursively as follows. Suppose for $PT_{n-1}$, the isomorphism is given by $t_{i_{1}...i_{n-4}}\mapsto\sum_{\sigma\in B_{i_{1}...i_{n-4}}}z_{\sigma}$. Then for $PT_{n}$ the isomorphism is
\[
t_{i_1 \ldots i_{n-3}} \mapsto \sum_{\sigma \in B_{i_1 \ldots i_{n-4}}} \shuffle_{i_{n-3}+1} z_\sigma.
\]
\end{theorem}

We delay the proof of this theorem and first state some of its immediate corollaries. 

\begin{corollary}
    The ideal of the Parke--Taylor variety $\mathrm{PT}_n$ is quadratically generated.
\end{corollary}   
\begin{proof}
    This follows from \cite[Theorem 8.6]{keel2009equations} and Theorem \ref{thm:log-canon-Parke--Taylor}.
\end{proof}

\begin{remark}
While both $\lc_n$ and $\pt_n$ are quadratically generated embeddings of $\monbar$, we will see in the following sections that the ways in which the quadratic generators of their respective ideals arise are quite different from an algebraic perspective. In particular we will show that our parametrization $\varphi_n$ of $\pt_n$ can also be seen as the composition of a monomial map $\tilde{\varphi}_n$ which we study in Section \ref{sec:toric} and the Pl\"ucker embedding whereas the log canonical embedding $\Phi_n$ is the composition of the Segre embedding and the iterated Kapranov embedding $\Omega_n$. From this perspective one can see that
\[
\ci(\pt_n) = (\tilde{\varphi_n}^\ast)^{-1}(\ci(\gr(2,n))), ~~ \ci(\lc_n) = (\xi^\ast)^{-1}(\ker(\Omega_n^\ast)).
\]
In the case of $\pt_n$, the monomial map $\tilde{\varphi}_n$, and its associated toric ideal are quite complicated as we will see in the next section, but the ideal $\ci(\gr(2,n))$, which we must compute the preimage of, is quite simple and it is well known that it is generated by the Pl\"ucker relations. On the other hand, for $\lc_n$, the Segre embedding is an extremely simple monomial map and its corresponding toric ideal is generated by quadratic binomials, however $\ker(\Omega_n^\ast)$ is more complicated and is generated by the cubics known as the \emph{Monin--Rana} equations \cite{gillespie2022proof, monin2017equations}. 
\end{remark}

For our next result, recall that the multinomial coefficient $\binom{n-3}{i_1\ldots i_{n-3}}$ is the number of ways to partition \(n-3\) distinct items into \(n-3\) distinct groups of nonnegative sizes \(i_1, i_2, \ldots, i_{n-3}\) such that $i_1+\ldots+i_{n-3}=n-3$. For the definition of the  \emph{asymmetric multinomial coefficient} $\abinom{n}{i_{n-3}\ldots i_{1}}$, see \cite[Definition 4.11]{cavalieri2021projective}.
\begin{proposition}
The degree of $\lc_n$, and thus $\pt_n$, is
$$\sum\limits_{i_1+\ldots+i_{n-3}=n-3} \binom{n-3}{i_1\ldots i_{n-3}}\abinom{n-3}{i_{n-3}\ldots i_{1}},$$
where $\binom{n-3}{i_1\ldots i_{n-3}}$ are the multinomial coefficients and $\abinom{n}{i_1\ldots i_{r}}$ are the asymmetric multinomial coefficients. 
\end{proposition}
 
\begin{proof}
    The coefficient $d_{i_1,\ldots,i_{n-3}}$ of the multidegree polynomial of $\Omega_n(\monbar)$ is the asymmetric multinomial coefficient $\abinom{n-3}{i_{n-3}\ldots i_{1}}$ by \cite[Theorem 1.1]{cavalieri2021projective}. Let $\xi$ denote the Segre embedding $\mathbb{P}^1\times\ldots\times \mathbb{P}^{n-3}\hookrightarrow \mathbb{P}^{(n-2)!-1}$.  Write $h_i=c_1(\mathcal{O}_{\mathbb{P}^i}(1))$, where 
    $c_1$ is the first Chern class. Then we have $c_1(\xi^*\mathcal{O}_{\mathbb{P}^{(n-2)!-1}}(1))=h_1+\ldots+h_{n-3}$ (see e.g. \cite[Equation (3.5)]{kozhasov2023minimal}). The degree of $\mathrm{LC}_n$ is given by 
    \begin{multline*}
        \deg \mathrm{LC}_n = \int\limits_{\mathbb{P}^{(n-2)!-1}} c_1(\xi^*\mathcal{O}_{\mathbb{P}^{(n-2)!-1}}(1))^{n-3}\cap[\mathrm{LC}_n] = \int\limits_{\mathbb{P}^{(n-2)!-1}}(h_1+\ldots+h_{n-3})^{n-3}\cap [\mathrm{LC}_n] =\\ =\sum\limits_{i_1+\ldots+i_{n-3}=n-3} \binom{n-3}{i_1\ldots i_{n-3}}d_{i_1\ldots i_{n-3}}=\sum\limits_{i_1+\ldots+i_{n-3}=n-3} \binom{n-3}{i_1\ldots i_{n-3}}\abinom{n-3}{i_1\ldots i_{n-3}}.\end{multline*}
Since $\mathrm{PT}_n$ is linearly isomorphic to $\mathrm{LC}_n$ by Theorem \ref{thm:log-canon-Parke--Taylor}, its degree is the same. 
\end{proof}

\begin{remark}
    For $n=5,6,7$ the degree of $\mathrm{PT}_n$ is $5,61,1379$ respectively. While the present paper was under review, the reference \cite{blankers2025} appeared, in which these same numbers appear in Section 5.6 along with a combinatorial interpretation.
\end{remark}

We now give an example for $n = 5$ which serves as the base case for the iterative construction described in the statement of Theorem \ref{thm:log-canon-Parke--Taylor}.

\begin{example}
\label{ex:LC-to-PT-5}
Let $\pt_5$ and $\lc_5$ be the closed Parke--Taylor variety and the log canoncial embedding of $\overline{\mathcal{M}}_{0.n}$ respectively. Then the linear map $L_5$ from $\lc_5$ to $\pt_5$ is given by:
\begin{align*}
L_5(t_{11}) &= z_{12345} + z_{12354} + z_{12435} + z_{12453} + z_{12534} + z_{12543} = \sum_{\substack{\sigma \\ \inv(\sigma) \subseteq \{(42,31), (53,31), (53,42)\}}} z_\sigma,\\
L_5(t_{12}) &= z_{12345} + z_{12354} + z_{12435}  = \sum_{\substack{\sigma \\ \inv(\sigma) \subseteq \{(42,31), (53,42)\}}} z_\sigma,\\
L_5(t_{13}) &= z_{12345} + z_{12435} + z_{12453} = \sum_{\substack{\sigma \\ \inv(\sigma) \subseteq \{(42,31), (53,31) \}}} z_\sigma, \\ 
L_5(t_{21}) &= z_{12345} + z_{12354} + z_{12534} = \sum_{\substack{\sigma \\ \inv(\sigma) \subseteq \{(53,31), (53,42)\}}} z_\sigma,\\
L_5(t_{22}) &= z_{12345} + z_{12354} = \sum_{\substack{\sigma \\ \inv(\sigma) \subseteq \{(53,42)\}}} z_\sigma,\\
L_5(t_{23}) &= z_{12345} = \sum_{\substack{\sigma \\ \inv(\sigma) \subseteq \emptyset}} z_\sigma.\\
\end{align*}
where all sums here are over $\sigma \in \Sigma_5$. 
Furthermore, this linear map sends the vanishing ideal $\ci(\lc_5)$ of the log canonical embedding to $\ci(\pt_5)$. Since $L_5$ is a linear map, the resulting ideal is also generated by 5 quadratics:
\begin{align*}
\ci(\pt_5) = 
\langle 
z_{12354}z_{12435}+z_{12345}z_{12453}+z_{12354}z_{12453}, z_{12345}z_{12534}+z_{12345}z_{12543}+z_{12354}z_{12543},\\
z_{12354}z_{12435}+z_{12345}z_{12534}+z_{12435}z_{12534}, z_{12345}z_{12453}+z_{12345}z_{12543}+z_{12435}z_{12543},\\
z_{12435}z_{12534}+z_{12453}z_{12534}+z_{12435}z_{12543}
\rangle.
\end{align*}
Note that the 5 quadratics above are not exactly the images of the 5 generators of $\ci(\lc_5)$ which appear in Example \ref{ex:log-canon-5} but they generate the same ideal. We will provide a detailed explanation of the quadratics which appear above in Section \ref{sec:idealPT} but here we note
that by the same theorem of \cite{buchsbaum1977algebra}, $\ci(\pt_5)$ can also be generated by the submaximal Pfaffians of a $5 \times 5$ antisymmetric matrix. In this case the following matrix suffices:
\begin{align*}
M_\pt = \begin{pmatrix}
0& -z_{12543}& -z_{12345}& -z_{12354}-z_{12534}& z_{12345} \\
z_{12543}& 0& -z_{12435}+z_{12543}& -z_{12543}& z_{12435}+z_{12453} \\
z_{12345}& z_{12435}-z_{12543}& 0& -z_{12534}& -z_{12435} \\ 
z_{12354}+z_{12534}& z_{12543}& z_{12534}& 0& 0 \\
-z_{12345}& -z_{12435}-z_{12453}&z_{12435}& 0& 0
\end{pmatrix}. 
\end{align*}
\end{example}

At first glance, one might hope that the linear map $L_5$ corresponds to a sum over permutations in a \emph{principal} lower order ideal in the \emph{right weak order} on the symmetric group $\Sigma_n \cong S_{n-2}$ \cite{bjornerbrenticoxeter}, but this is unfortunately not the case. However, each coordinate $t_{i_1 \ldots i_n}$ does correspond to a sum over permutations in a lower order ideal in the right weak order. We will expand upon this later in this section but we first develop a few results which will be core ingredients in the proof of Theorem \ref{thm:log-canon-Parke--Taylor}.

\begin{lemma}
\label{lemma:plucker-sum-shuffle}
Let $p_{ij}$ denote the Pl\"ucker coordinates on $\gr(2, n)$. Then for any $i$ such that $3 \leq i \leq n-1$ and $\sigma \in \Sigma_{n-1}$, it holds that
\[
\sum_{\substack{i \leq j \leq n-1}} \frac{p_{\sigma_j \sigma_{j + 1}}}{p_{\sigma_j n}p_{\sigma_{j +1} n}} = -\frac{p_{1 \sigma_i}}{p_{\sigma_i n} p_{1n}}.
\]
\end{lemma}
\begin{proof}
The statement follows from iteratively finding common denominators and applying the Pl\"ucker relations to the left hand side. Here Pl\"ucker coordinates are treated as antisymmetric in their indices. Observe that the sum of the last two terms in the above sum yields
\[
-\frac{p_{1 \sigma_{n-1}}}{p_{1n} p_{\sigma_{n-1} n}} +  \frac{p_{\sigma_{n-2} \sigma_{n-1}}}{p_{\sigma_{n-2} n } p_{\sigma_{n-1} n}} = \frac{-p_{1 \sigma_{n-1}} p_{\sigma_{n- 2} n} + p_{1 n} p_{\sigma_{n- 2} \sigma_{n - 1}}}{p_{1n} p_{\sigma_{n-2} n } p_{\sigma_{n-1} n}} = \frac{-p_{1 \sigma_{n-2}} p_{\sigma_{n-1} n}}{p_{1n} p_{\sigma_{n-2} n } p_{\sigma_{n-1} n}} =  \frac{-p_{1 \sigma_{n-2}}}{p_{1n} p_{\sigma_{n-2} n }}.
\]
Repeatedly applying this technique to the output of the previous step together with the last remaining term of the original sum then yields the desired result. 
\end{proof}

\begin{corollary}
\label{cor:pt-recursive-shuffle}
Let $\varphi_n$ be the parametrization of $\pt_n$ as defined in Definition \ref{def:pt}. Then for any $3 \leq i\leq n-1$ and $\sigma \in \Sigma_{n-1}$, it holds that
\[
\varphi_{n-1}^\ast(z_\sigma) \frac{-p_{1 \sigma_i }}{p_{\sigma_i n} p_{1n}} = \varphi_n^\ast(\shuffle_{\sigma_i} z_\sigma),
\]
where $\shuffle_{\sigma_i} z_\sigma$ is the sum defined in Definition \ref{defn:shuffle-sum}. 
\end{corollary}
\begin{proof}
First, observe that
\[
\varphi_{n-1}^\ast(z_\sigma) \frac{p_{\sigma_i \sigma_{i + 1}}}{p_{\sigma_i n}p_{\sigma_{i +1} n}} = \varphi_n^\ast(z_\tau),
\]
where $\tau$ is the permutation obtained from $\sigma$ by inserting the new letter $n$ in between $\sigma_{i}$ and $\sigma_{i+1}$. Multiplying both sides of the equality in Lemma \ref{lemma:plucker-sum-shuffle} by $\varphi_{n-1}^\ast(z_\sigma)$ then yields
\[
\varphi_{n-1}^\ast(z_\sigma)\frac{-p_{1 \sigma_i}}{p_{\sigma_i n} p_{1n}} = \varphi_{n-1}^\ast(z_\sigma)
\sum_{\substack{i \leq j \leq n-1}} \frac{p_{\sigma_j \sigma_{j + 1}}}{p_{\sigma_j n}p_{\sigma_{j +1} n}} = \varphi_n^\ast(\shuffle_{\sigma_i} z_\sigma),
\]
where the second equality holds since the set of permutations $\tau$ obtained by inserting $n$ between $\sigma_i$ and $\sigma_j$ is exactly the same as the set of $\tau \in \Sigma_n$ such that $\tau \setminus \{n\} = \sigma$ and $(\sigma_i, n) \notin \inv(\tau)$.

\end{proof}

\begin{remark}
\label{rem:alternate-pt-shuffle}
Note that the above corollary also holds if one uses the restriction of $\varphi_n$ to $\cm_{0, n}$ as discussed in Remark \ref{rem:restrict-to-m0n} instead of the original definition in Definition \ref{def:pt}. In that case, the equality in the previous corollary becomes
\[
\varphi_{n-1}^\ast(z_\sigma) \frac{p_1 - p_{\sigma_i}}{(p_{n} - p_{\sigma_i}) (p_{n} - p_1)} = \varphi_n^\ast(\shuffle_{\sigma_i} z_\sigma).
\]
\end{remark}

We are now ready to complete the proof of Theorem \ref{thm:log-canon-Parke--Taylor}.

\begin{proof}[Proof of Theorem \ref{thm:log-canon-Parke--Taylor}]
Let $\varphi_n$ and $\Phi_n$ be the parametrizations of $\pt_n$ and $\lc_n$ respectively. We will construct a linear map $L_n: \cc[t_{i_1 \ldots i_{n-3}}] \to \cc[z_\sigma ~:~ \sigma \in \Sigma_n ]$ such that $\varphi_n^\ast(L_n(t_{i_1 \ldots i_{n-3}})) = \Phi_n^\ast(t_{i_1 \ldots i_{n-3}})$
and thus $L_n$ maps $\pt_n^\circ$ to $\lc_n^\circ$. Since both the open Parke--Taylor variety and log canonical embedding are of course Zariski dense in their corresponding closures, this implies that $L_n$ maps $\pt_n$ to $\lc_n$ as required. 

Suppose by way of induction that we already have a linear isomorphism $L_{n-1}$ between $\lc_{n-1}$ and $\pt_{n-1}$ of the required form and define $L_n$ by
\[
L_n(t_{i_1 \ldots i_{n-3}}) = \sum_{\sigma \in B_{i_1 \ldots i_{n-4}}} \shuffle_{i_{n-3} + 1} z_\sigma, 
\] 
where $B_{i_1 \ldots i_{n-1}}$ is defined by $L_{n-1}(t_{i_1 \ldots i_{n-4}}) = \sum_{\sigma \in B_{i_1 \ldots i_{n-4}}} z_\sigma$. Since $\Phi_n$, is the composition of the iterated Kapranov embedding and the Segre embedding, we have that
\begin{equation}
\label{eqn:recursive-log-canon}
\Phi_n^\ast(t_{i_1 \ldots i_{n-3}})= \Phi_{n-1}^\ast(t_{i_1 \ldots i_{n-4}}) \frac{p_1 - p_{i_{n-3} + 1}}{p_n - p_{i_{n-3} + 1}} = 
\Phi_{n-1}^\ast(t_{i_1 \ldots i_{n-4}}) \frac{p_1 - p_{i_{n-3} + 1}}{(p_n - p_{i_{n-3} + 1})(p_n - p_1)}
\end{equation}
Where the last equality simply corresponds to the fact that the domain of the Segre embedding is the product of projective spaces $\pp^1 \times \pp^2 \times \cdots \times \pp^{n-3}$ and thus we may scale the last component by $\frac{1}{p_n - p_1}$ without changing the corresponding point in projective space. By our induction hypothesis, we have that
\begin{equation*}
\Phi_{n-1}^\ast(t_{i_1 \ldots i_{n-4}}) = \varphi_{n-1}^\ast(L_{n-1}(t_{i_1 \ldots i_{n-4}})) = \varphi_{n-1}^\ast \left(\sum_{\sigma \in B_{i_1 \ldots i_{n-4}}} z_\sigma \right)
\end{equation*}
and combining this equation with Equation \ref{eqn:recursive-log-canon} yields
\[
\Phi_n^\ast(t_{i_1 \ldots i_{n-3}}) = \varphi_{n-1}^\ast \left(\sum_{\sigma \in B_{i_1 \ldots i_{n-4}}} z_\sigma \right) \frac{p_1 - p_{i_{n-3} + 1}}{(p_n - p_{i_{n-3} + 1})(p_n - p_1)} = \varphi_n^\ast \left( \sum_{\sigma \in B_{i_1 \ldots i_{n-4}}} \shuffle_{i_{n-3}+ 1} z_\sigma \right),
\]
where the last equality follows immediately by applying Corollary \ref{cor:pt-recursive-shuffle}, or more specifically the alternative version of it stated in Remark \ref{rem:alternate-pt-shuffle}. Since the right hand side of the above equation is $\varphi_n^\ast(L_n(t_{i_1 \ldots i_{n-3}}))$ by definition, this shows that $\Phi^\ast(t_{i_1 \ldots i_{n-3}}) = \varphi_n^\ast(L_n(t_{i_1 \ldots i_{n-3}}))$. 

It remains to show that $L_n$ is invertible. To do this, we inductively show that each column of the matrix representing $L_n$ contains a pivot, so that the matrix can be made triangular by a permutation of rows. Suppose that the matrix representing $L_{n-1}$ is triangular, and that the pivot in $L_{n-1}$ of column $t_{i_1 \ldots i_{n-4}}$ occurs in the row indexed by $\sigma \in \Sigma_{n-1}$. Then we claim that the column of $L_n$ indexed by $t_{i_1 \ldots i_{n-3}}$ has a pivot in the row indexed by $\sigma'$ which is obtained from $\sigma$ by inserting $n$ directly after $i_{n-3}+1$. Order the rows of $L_n$ in blocks corresponding to $\sigma \in \Sigma_{n-1}$ (in the order that makes the matrix of $L_{n-1}$ triangular) such that the first element of the block has $n$ inserted in the rightmost position and the last element has $n$ inserted in the leftmost position. Then since $\sigma$ was a pivot for $t_{i_1 \ldots i_{n-4}}$ in $L_{n-1}$, by construction of our map,  if $\tau \in \Sigma_{n-1}$ appears after $\sigma$ 
then all entries of the block corresponding to $\tau$ in the $t_{i_1 \ldots i_{n-3}}$ column of $L_n$ must be zero. The $\sigma'$ entry of that same column will of course have a one and be the last entry of the last non-zero block and thus is indeed a pivot of $L_n$. This is true in the case of $n = 5$ which completes the proof. 


\end{proof}

\begin{remark}
    The Parke--Taylor functions have a natural connection to the geometry of $\overline{\mathcal{M}}_{0,n}$. The space $\overline{\mathcal{M}}_{0,n}$ is isomorphic to the wonderful model of the braid arrangement, and the Parke--Taylor functions are the canonical functions of the regions of this model in the sense of \emph{positive geometry} \cite[Section 5]{wondertopes}. Thus, Theorem \ref{thm:log-canon-Parke--Taylor} connects log canonical embeddings and canonical forms of positive geometries. This connection seems to be rather general, and we will study it further in future work. 
\end{remark}

The following example showcases the iterative construction of the map $L_n$. 

\begin{example}
\label{ex:iterative-LC-to-PT}
Let $n = 6$ and recall from Example \ref{ex:LC-to-PT-5} that $L_5(t_{21}) = z_{12345} + z_{12354} + z_{12534}$. Then $i_{n-3} = 2$ and so $L_6(t_{212})$ is given by
\begin{align*}
L_6(t_{212}) &= \sum_{\sigma \in B_{21}} \shuffle_{3} z_\sigma =  \underline{\textcolor{red}{\shuffle_3 z_{12345}}} + \underline{\underline{\textcolor{blue}{\shuffle_3 z_{12354}}}}+ \textcolor{ForestGreen}{\shuffle_3 z_{12534}}\\
&= \underline{\textcolor{red}{z_{123456} + z_{123465} + z_{123645}}} + \underline{\underline{\textcolor{blue}{z_{123546} + z_{123564} + z_{123654}}}}  + \textcolor{ForestGreen}{z_{125346} + z_{125364}}, 
\end{align*}
where the colored (or underlined) terms in the second line are the corresponding expansion of those in the first. As explained in the proof of Corollary \ref{cor:pt-recursive-shuffle}, note that $\shuffle_{3} z_\sigma$ can also be thought of as a sum over all permutations $\tau$ such that $\tau$ is obtained from $\sigma$ by inserting $6$ in to $\sigma$ at any position after $3$. 
\end{example}

The following result clarifies the relationship between the support sets $B_{i_1 \ldots i_{n-3}}$, which encode the linear map $L_n$ between $\pt_n$  and $\lc_n$, and the \emph{right weak order} on permutations.

\begin{proposition}
The set $B_{i_1, \cdots i_{n-3}}$ is a lower order ideal in the right weak order. Moreover, $B_{i_1, \ldots i_{n-3}} = \{\sigma \in \Sigma_n ~:~ \inv(\sigma) \subseteq S \cup T\}$ where
$S=\bigcup\limits_{\sigma^{\prime}\in B_{i_{1}...i_{n-4}}}\mathrm{inv}(\sigma^{\prime})$ and $T = \{(i, n) ~:~ \text{there exists } \tau \in B_{i_1 \ldots i_{n-4}}, ~~\tau^{-1}(i) > \tau^{-1}(i_{n-3} + 1) \}$.
\end{proposition}
\begin{proof}
We first show that $B_{i_1, \cdots i_{n-3}}$ is indeed a lower order ideal in the weak order by induction on $n$. Observe that for $n = 5$ this is indeed the case and assume the result holds for $B_{i_1 \ldots i_{n-4}}$. 
Now suppose that $\sigma \in B_{i_1 \ldots i_{n-3}}$ and $\tau \leq \sigma$ in the right weak order on $\Sigma_{n}$. Note that $n$ is to the right of $i_{n-3}+1$ in $\sigma$ and the same is true for $\tau$, since $\mathrm{inv}(\tau)\subseteq \mathrm{inv}(\sigma)$. So we must show that $\tau \in B_{i_1 \ldots i_{n-3}}$. Let $\tau' = \tau \setminus \{n\} \in \Sigma_{n-1}$ be the permutation obtained from $\tau$ by deleting $n$ and define $\sigma'$ analogously. Then by definition of the weak order, it must be that $\tau' \leq \sigma'$ and since $\sigma \in B_{i_1 \ldots i_{n-3}}$, it must be that $\sigma' \in B_{i_1 \ldots i_{n-4}}$. 
By the inductive hypothesis, $\tau' \in B_{i_1 \ldots i_{n-4}}$ and thus we see that $\tau \in B_{i_1 \ldots i_{n-3}}$ which shows that $B_{i_1 \ldots i_{n-3}}$ is indeed a lower order ideal. 

To prove the second statement, we also proceed by induction with $n = 5$ serving as the base case. Let $S$ and $T$ be defined as above and let $C_{i_1 \ldots i_{n-3}} = \{\sigma \in \Sigma_n ~:~ \inv(\sigma) \subseteq S \cup T\}$.
Note that $B_{i_1, \ldots i_{n-4}} = \{\sigma' \in \Sigma_{n-1} ~:~ \inv(\sigma') \subseteq S\}$.
Clearly it holds that $B_{i_1 \ldots i_{n-3}} \subseteq C_{i_1 \ldots i_{n-3}}$ so it remains to show the reverse inclusion. So suppose that $\sigma \in C_{i_1 \ldots i_{n-3}}$ and observe that $\sigma' = \sigma \setminus \{n\} \in B_{i_1 \ldots i_{n-4}}$ since $\inv(\sigma') = \inv(\sigma) \setminus T \subseteq S$. But this once again immediately implies that $\sigma \in B_{i_1 \ldots i_{n-3}}$ since $(i_{n-3}+1, n) \notin \inv(\sigma)$.  
\end{proof}

The fact that $B_{i_1 \ldots i_{n-3}}$ can be described as a set of permutations whose inversion sets are all contained in a given set of inversions is actually quite special and does not hold for every lower order ideal of the weak order. For example if one considers the lower order ideal in the weak order on $S_3$ generated by the permutations $231$ and $312$ as pictured in Figure \ref{fig:weak-order}, then $\inv(321) = \inv(231) \cup \inv(312) = \{(1,2), (1,3), (2,3)\}$ however $321 \notin \langle 231, 312 \rangle$.

\section{Binomial Parke--Taylor identities} \label{sec:toric}

The references \cite{maclagan2010chow} and \cite{keel2009equations} describe the equations of $\overline{\cm}_{0,n}$ in Cox rings of toric varieties into which $\overline{\cm}_{0,n}$ is embedded. In the former reference, the toric variety has dimension $\binom{n}{2}-n$ and is associated with the space of phylogenetic trees, while in the latter it is the product of projective spaces $\mathbb{P}^1\times\ldots\times \mathbb{P}^{n-3}$. In this section we describe yet another (projective) toric variety that contains the Parke--Taylor variety $\mathrm{PT}_n$. We denote this variety by $T_n$ and call it the \emph{toric Parke--Taylor variety}. It is naturally obtained by extending the Parke--Taylor map $\varphi_n$ from $\mathrm{Gr}(2,n)$ to the map $\Tilde{\varphi}_n$ on the ambient space $\mathbb{P}^{\binom{n}{2}-1}$ and taking the closure of the image. We start by recalling the definition of a toric variety. 

\begin{definition}
Let $A \in \zz^{d \times r}$ be an integer matrix and let $a_j$ denote the $j$-th column of $A$. Then the \emph{affine toric variety} $V_A$ associated to the matrix $A$ is the Zariski closure of the image of the map
\begin{align*}
\varphi_A : \cc^d &\to \cc^r, \\
             t  &\mapsto (t^{a_1}, t^{a_2}, \ldots, t^{a_r}),
\end{align*}
where $t^{a_j} = \prod_{i = 1}^d t_i^{a_{ij}}$. That is, $V_A = \overline{\mathrm{image}(\varphi_A)}$. The \emph{toric ideal} associated to $A$ is denoted $I_A = \mathcal{I}(V_A)$. If $A$ contains a row of all ones, then the variety $V_A$ is an affine cone over some projective variety. We call this variety the \emph{projective toric variety} associated to $A$. 
\end{definition}

Every toric ideal is generated by binomials \cite[Chapter 4]{sturmfels1996groebner}, and thus the defining equations of $T_n$ produce a family of binomial equations that also hold on the smaller variety $\mathrm{PT}_n$. The following result gives the dimension of $T_n$. 

\begin{proposition} \label{prop:toricdim}
    The affine cone over the toric variety $T_n$ is an irreducible variety of dimension $\binom{n-1}{2}-1 = \binom{n}{2}-n$. The dimension of $T_n$ is thus $\binom{n}{2}-n-1$. 
\end{proposition} 
\begin{proof}
Note that the map $\tilde{\varphi}_n$ also parametrizes the affine cone over $T_n$ when considered as a map from $\cc^{\binom{n}{2}} \to \cc^{(n-2)!}$. 
Thus our strategy is to indicate an affine space of dimension $\binom{n-1}{2}-1$ in $\cc^{\binom{n}{2}}$ on which $\tilde{\varphi}_n$ is dominant and invertible. 

    
    Consider the affine space $W\subset\mathbb{C}^{\binom{n}{2}}$ defined by the $n$ independent equations $p_{12}=1$ and $p_{in}=1$ for $i\in[n-1]$. The coordinates on $W$ are then $p_{ij}$ with $1\leq i<j<n$ and $(i,j)\neq(1,2)$ . As shown in the proof of Proposition \ref{thm:PTM0n}, any cross-ratio $[ij|kl] = \dfrac{p_{ik}p_{jl}}{p_{jk}p_{il}}$ is a ratio of two Parke--Taylor functions. On $W$ we have $[1n|2j]=\dfrac{1}{p_{1j}}$ for $2<j<n$ and $\dfrac{[1n|2j]}{[1i|jn]}=p_{ij}$ for $1<i<j<n$. Thus, every $p_{ij}$ is a ratio of polynomial combinations of Parke--Taylor functions, and $\tilde{\varphi}_n|_W$ is generically invertible. Since $\dim W=\binom{n-1}{2}-1$, the proposition follows.
\end{proof}

\begin{remark}
    The toric variety $T_n$ is different from the variety $X_\Delta$ considered in \cite{maclagan2010chow}, since their dimensions differ by one. One can also consider the toric variety parametrized a larger set of $\dfrac{(n-1)!}{2}$ Parke--Taylor functions that are independent modulo the ``obvious'' linear relations (cyclic symmetry and flips described in Section \ref{sec:prelim}). This variety is birational to $X_\Delta$, since both contain the intrinsic torus of $\mathcal{M}_{0,n}$ as a Zariski-dense subset. However there is still no isomorphism: unlike $X_\Delta$, this bigger toric Parke--Taylor variety is projective and contains toric strata that miss $\mathcal{M}_{0,n}$. 
\end{remark}


For the toric Parke--Taylor  variety $T_n$, the corresponding matrix is $A_n \in \zz^{\binom{n}{2} \times (n-2)!}$ whose rows are naturally indexed by subsets $\{i, j\}\subset[n]$ and columns are indexed by permutations $\sigma \in \Sigma_n$ with entries given by
\[
(A_n)_{\{i, j\}, \sigma} =
\begin{cases}
1, \mbox{ if } i \mbox{ and } j \mbox{ are adjacent in } \sigma \\
0, \mbox{ otherwise}
\end{cases}
\]
where adjacencies are pairs of the form $\sigma_{i}\sigma_{i+1}$ for $i \in [n]$ taken cyclically. 
It is then immediate from the definition that $\tilde{\varphi}_n = \varphi_{A_n}$. The ideal of $T_n$ can then be described by using the following result.

\begin{proposition}\cite[Chapter 4]{sturmfels1996groebner}
Let $A \in \zz^{d \times r}$ be an integer matrix and $I_A$ be the associated toric ideal. Then
\[
I_A = \langle z^u - z^v ~:~ u, v \in \zz_{\geq 0}^r \mbox{ and } Au = Av \rangle.
\]
\end{proposition}

In other words, the toric ideal gives an encoding of the integer kernel, $\ker_\zz(A)$. Every binomial $x^{u} - x^v \in I_A$ gives rise to a vector $u - v \in \ker_\zz(A)$ and any vector $u \in \ker_\zz(A)$ can be written as $u = u^+ - u^-$ for vectors $u^+, u^- \in \zz_{\geq 0}^r$ which naturally corresponds to a binomial $x^{u^+} - x^{u^-} \in I_A$. In our Parke--Taylor setting, the variables parametrizing $T_n$ are labeled by adjacencies of permutations on $[n]$, which immediately yields following result.

\begin{proposition}
\label{prop:toric-gens-via-adjs}
A binomial equation of the form 
$$\prod\limits_{\sigma\in A} z_\sigma - \prod\limits_{\tau\in B} z_\tau = 0$$
holds on $T_n$ if and only if the multiset of adjacencies of all permutations in $A$ (counted with multiplicity) is the same as the multiset of adjacencies of all permutations in $B$. Moreover, the set of all such binomials above generates $\ci(T_n)$. 
\end{proposition}

This simple combinatorial rule immediately allows us to construct many nonlinear Parke--Taylor relations. However, this set of equations is very large, and the equations have unbounded degree. It is therefore desirable to find a simpler description of the ideal of $T_n$. A natural hope is that a basis of $\mathrm{ker}_{\mathbb{Z}}(A_n)$ produces such a generating set. Unfortunately, in general this is not the case, as we show in Example \ref{example:toric-6}. However, such a basis is enough to describe the equations of the dense torus of $T_n$ inside that of $\mathbb{P}^{(n-2)!-1}$. In other words, it is enough to describe the ideal of $T_n$ up to saturation by the product of the variables.  
\begin{proposition}\cite[Lemma 12.2]{sturmfels1996groebner}
Let $\mathcal{B} = \{u_1, \ldots, u_n\}$ be a basis for $\ker_\zz(A)$. Then
\[
I_A = \langle x^{u^+} - x^{u^-} ~:~ u \in \mathcal{B} \rangle : x_1x_2 \cdots x_n^\infty
\]
where $u = u^+ - u^-$ is a decomposition of $u \in \zz^r$ such that $u^+, u^- \in \zz_{\geq 0}^r$. 
\end{proposition}

From the point of view of applications, our goal is to characterize relations between the Parke--Taylor functions, and these are always non-zero.
Therefore, we are primarily interested in the relations that describe the intersection of $\mathrm{PT}_n$ with the torus, and these are precisely encoded by a spanning set of $\mathrm{ker}_{\mathbb{Z}}(A_n)$.

\begin{example}
\label{example:toric-5}
Consider the toric Parke--Taylor variety $T_5$ whose associated integer matrix is
\[
A_5 = 
\begin{blockarray}{ccccccc}
& 12345 & 12354 & 12435 & 12453 & 12534 & 12543 \\
\begin{block}{c(cccccc)}
p_{12} &1&1&1&1&1&1\\
p_{13} &0&0&0&1&0&1\\
p_{14} &0&1&0&0&1&0\\
p_{15} &1&0&1&0&0&0\\
p_{23} &1&1&0&0&0&0\\
p_{24} &0&0&1&1&0&0\\
p_{25} &0&0&0&0&1&1\\
p_{34} &1&0&1&0&1&1\\
p_{35} &0&1&1&1&1&0\\
p_{45} &1&1&0&1&0&1\\
\end{block}
\end{blockarray}.
\]

Since $\rank(A_5) = \dim(T_n) = 5$, we know that $\ci(T_n)$ must be a principal ideal, meaning an ideal which is generated by a single polynomial. Thus, the kernel of $A_5$ is one-dimensional. It is spanned by the vector 
\[
u = (1, -1, -1, 1, 1, -1) = (1, 0, 0, 1, 1, 0) - (0, 1, 1, 0, 0, 1) = u^+ - u^-.
\]
This vector corresponds to the binomial
\[
z^{u^+} - z^{u_{-}} = z_{12345}z_{12453}z_{12534} - z_{12354}z_{12435}z_{12543}.
\]
Observe that the multiset of adjacencies $\{12^3, 13, 14, 15, 23, 24, 25, 34^2, 35^2, 45^2\}$ is indeed the same in each term of this binomial. 
\end{example}


\begin{example}
\label{example:toric-6}
The toric ideal $\ci(T_6) \subseteq \kk[z_\sigma ~:~ \sigma \in \Sigma_6]$ is minimally generated by 24 quadratics, 164 cubics, and 6 quartics which yields a total of 194 minimal generators. However, $A_6 \in \zz^{15 \times 24}$ is rank 9 and thus the kernel is minimally generated by only 15 integer vectors. The integer kernel can be generated by 15 vectors which encode only quadratic and cubic relations. In particular the following polynomials correspond to one choice of basis for $\ker_\zz(A_n)$:
\begin{align*}
z_{123465}z_{124536}-z_{123456}z_{124635}, z_{123546}z_{124365}-z_{123456}z_{124635},\\
z_{123456}z_{125364}-z_{123546}z_{125634}, z_{123456}z_{126453}-z_{123546}z_{126543},\\ 
z_{123456}z_{125364}-z_{123564}z_{125436}, z_{123564}z_{126345}-z_{123465}z_{126354},\\
z_{123645}z_{124356}-z_{123456}z_{124635}, z_{123465}z_{125463}-z_{123645}z_{125643},\\
z_{123465}z_{126354}-z_{123645}z_{126534}, z_{123564}z_{124563}-z_{123654}z_{124653},\\
z_{123654}z_{125346}-z_{123456}z_{125364}, z_{123465}z_{126354}-z_{123654}z_{126435},\\
z_{124563}z_{125346}-z_{124356}z_{125463}, z_{124563}z_{126435}-z_{124365}z_{126453},\\
z_{126345}z_{126453}z_{126534}-z_{126354}z_{126435}z_{126543}.
\end{align*}
However, the quartic 
\[
z_{123654}z_{124536}z_{125463}z_{126345} - z_{123645}z_{124563}z_{125436}z_{126354}
\]
does not lie in the ideal generated by the 15 polynomials above despite the fact that it is a minimal generator of $\ci(T_6)$. 
\end{example}
Example \ref{example:toric-6} shows that not only does a basis of $\ker_\zz(A_n)$ generally not suffice to generate $\ci(T_n)$, but also that there may exist minimal generators of high degree. 

We now describe a simple set of binomial equations that hold on $T_n$. We conjecture that this set forms a spanning set of $\mathrm{ker}_{\mathbb{Z}}(A_n)$ (Conjecture \ref{conj:toric-quads-gen-ker}). First, we introduce two important notations which will be useful throughout the rest of this paper. For a permutation $\sigma = \sigma_1 \sigma_2 \cdots \sigma_n$ in one-line notation, we denote the adjacency of $\sigma_i$ and $\sigma_{i+1}$ by $\sigma_{i}\sigma_{i+1}$.
Next, in the following results we will often use the \emph{tableau notation} of a binomial in which we write
\[
\prod\limits_{\sigma\in A} z_\sigma - \prod\limits_{\tau\in B} z_\tau = 
\begin{bmatrix}
1 & 2 & \cdots & \sigma_i^{(1)} &  \cdots & \sigma_n^{(1)} \\
1 & 2 & \cdots & \vdots         & \cdots & \vdots  \\
1 & 2 & \cdots & \sigma_i^{(m)} & \cdots & \sigma_n^{(m)} \\
\end{bmatrix} - 
\begin{bmatrix}
1 & 2 & \cdots & \tau_i^{(1)} & \cdots & \tau_n^{(1)} \\
1 & 2 & \cdots & \vdots       & \cdots & \vdots  \\
1 & 2 & \cdots & \tau_i^{(m)} & \cdots & \tau_n^{(m)} \\
\end{bmatrix}. 
\]
where $A = \{\sigma^{(1)}, \ldots, \sigma^{(m)}\}$ and $\sigma^{(\ell)} = 12\sigma_{3}^{(\ell)}\sigma_{4}^{(\ell)}\cdots\sigma_{n}^{(\ell)}$ and $B$ is defined analogously. Each row of each matrix therefore encodes a permutation in $S_n$.

\begin{theorem} \label{thm:toricquadratics}
Let $n \geq 6$, $\{i, j, k, l\} \subset [n]$, and $\Pi_1 | \Pi_2 | \Pi_3 $ be an ordered partition of $[n] \setminus \{1, 2, i, j, k, l\}$. Then the following quadratic binomials are all contained in $\ci(T_n)$:
\begin{itemize}
\item $\begin{bmatrix}
1 & 2 & \alpha & i & j & k & l & \gamma \\
1 & 2 & \alpha & j & l & i & k & \gamma \\ 
\end{bmatrix} - 
\begin{bmatrix}
1 & 2 & \alpha & i & j & l & k & \gamma \\
1 & 2 & \alpha & j & k & i & l & \gamma \\ 
\end{bmatrix},$\\
\item $\begin{bmatrix}
1 & 2 & \alpha & i & j & \beta & k & l & \gamma \\
1 & 2 & \alpha & j & i & \beta & l & k & \gamma \\ 
\end{bmatrix} - 
\begin{bmatrix}
1 & 2 & \alpha & i & j & \beta & l & k & \gamma \\
1 & 2 & \alpha & j & i & \beta & k & l & \gamma
\end{bmatrix}, ~\beta \neq \emptyset$,
\end{itemize}
where $\alpha, \beta, \gamma$ are treated as permutations in the symmetric group on the letters $\Pi_1, \Pi_2,$ and $\Pi_3$ respectively. 
\end{theorem}

The proof of this theorem is a straightforward application of the following combinatorial lemma and Proposition \ref{prop:toric-gens-via-adjs}.

\begin{lemma}
\label{lemma:lift-binomials}
Suppose that $A, B \subset \Sigma_n$ have the same multiset of adjacencies and $|A| = |B| = m$.
If there exists an index $i \in \{2,\ldots, n\}$ such that $\{\sigma_i \sigma_{i+1} ~:~ \sigma \in A \} = \{\tau_i \tau_{i+1} ~:~ \tau \in B \}$ as multisets, then for any $k \in \zz_{>0}$, it holds that the sets of permutations
\begin{equation}
\label{eqn:lift-binomial}
\begin{bmatrix}
1 & 2 & \cdots & \sigma_i^{(1)} & \delta & \sigma_{i + 1}^{(1)} & \cdots & \sigma_n^{(1)} \\
1 & 2 & \cdots & \vdots         & \delta & \vdots               & \cdots & \vdots  \\
1 & 2 & \cdots & \sigma_i^{(m)} & \delta & \sigma_{i + 1}^{(m)} & \cdots & \sigma_n^{(m)} \\
\end{bmatrix} \text{ and } 
\begin{bmatrix}
1 & 2 & \cdots & \tau_i^{(1)} & \delta & \tau_{i + 1}^{(1)} & \cdots & \tau_n^{(1)} \\
1 & 2 & \cdots & \vdots       & \delta & \vdots             & \cdots & \vdots  \\
1 & 2 & \cdots & \tau_i^{(m)} & \delta & \tau_{i + 1}^{(m)} & \cdots & \tau_n^{(m)} \\
\end{bmatrix}, 
\end{equation}
where $\delta = \delta_1 \delta_2 \cdots \delta_k$ is a permutation on the letters $n+1, \ldots, n+k$, have the same multiset of adjacenices.
\end{lemma}
\begin{proof}
The equality $\{ \sigma_i \sigma_{i+1} ~:~ \sigma \in A \} = \{\tau_i \tau_{i+1}~:~ \tau \in B \}$ of the multiset of adjacencies at position $i$ implies 
that the multisets
\[
\{ \sigma_i \delta_1,~ \delta_k \sigma_{i+1} ~~:~ \sigma \in A \} = \{ \tau_i \delta_1,~ \delta_k \tau_{i+1} ~:~ \tau \in B\}
\]
are also equal. For any permutation $\pi \in \Sigma_n$, let $\Adj(\pi)$ denote the set of adjacencies of $\pi$ taken cyclically and $\adj(\pi)$ denote the standard set of adjacencies (meaning taken acyclically). We have
\begin{equation}
\label{eqn:lift-perm-adj-decomp}
\Adj(12 \sigma_3 \cdots \sigma_i \delta \sigma_{i+1} \ldots \sigma_n) = \adj(12\ldots \sigma_i) \cup\{\sigma_i \delta_1, \delta_k \sigma_{i+1
} \} \cup \adj(\sigma_{i+1} \ldots \sigma_n 1).   
\end{equation}

Since $\cup_{\sigma \in A} \Adj(\sigma) = \cup_{\tau \in B} \Adj(\tau)$, by assumption it follows that $\cup_{\sigma \in A} \Adj(\sigma) \setminus \{\sigma_i \sigma_{i+1}\} = \cup_{\tau \in B} \Adj(\tau) \setminus \{\tau_i \tau_{i+1}\}$. Combining this equality together with the fact that
\begin{align}
\label{eqn:perm-union-decomp}
&\cup_{\sigma \in A} \Adj(\sigma) \setminus \{\sigma_i \sigma_{i+1} \} = \cup_{\sigma \in A} \adj(\sigma_1 \ldots \sigma_i) \cup \adj(\sigma_{i+1} \ldots \sigma_n 1), \\
&\cup_{\tau \in B} \Adj(\tau) \setminus \{\tau_i \tau_{i+1}\} = \cup_{\tau\in B} \adj(\tau_1 \ldots \tau_i) \cup \adj(\tau_{i+1} \ldots \tau_n 1)
\end{align}
yields that the multisets of adjacencies of the permutations in the two monomials in Equation \eqref{eqn:lift-binomial} are the same.
\end{proof}

\begin{proof}[Proof of Theorem \ref{thm:toricquadratics}]
    We prove this claim for the first type of binomials only, since the second type proceeds in the exact same way. Observe that by Proposition \ref{prop:toric-gens-via-adjs}, the multiset of adjacencies appearing in the permutations $\{12ijkl, 12jlik\}$ is $\{12^2, 2i, ij, jk, kl, l1, jl, li, ik, k1\}$ and is exactly equal to that of the pair $\{12ijlk, 12jkil\}$. This implies that the binomial $z_{12ijkl} z_{12jlik} - z_{12ijlk} z_{12jkil}$ lies in the ideal $\ci(T_6)$. Now observe that the indices $m = 2$ and $m = 6$ both satisfy the requirements of Lemma \ref{lemma:lift-binomials} in the sense that the multiset of adjacencies in each binomial is exactly the same at each of these indices. Thus, by Lemma \ref{lemma:lift-binomials}, we may insert the permutations $\alpha$ and $\gamma$ at positions $2$ and $6$ respectively to obtain a binomial which lies in $\ci(T_n)$. This completes the proof. 
\end{proof}

Lemma \ref{lemma:lift-binomials} is a useful tool for finding generators of $\ci(T_n)$ since it allows us to take many relations that hold on small Parke--Taylor varieties and lift them to relations which hold for any number of particles $n$.

\begin{example}
Recall that $\ci(T_5) = \langle z_{12345}z_{12453}z_{12534} - z_{12354}z_{12435}z_{12543} \rangle$ and observe that at $i = 2$, the multiset of adjacencies in each monomial is $\{23, 24, 25\}$. Thus by Lemma \ref{lemma:lift-binomials}, we can insert the permutation $\delta = 6 \in S_1$ after position $i = 2$ to get the binomial $z_{126345}z_{126453}z_{126534}-z_{126354}z_{126435}z_{126543} \in \ci(T_6)$. 
\end{example}

This type of \emph{lifting} operation is a powerful tool in combinatorial commutative algebra and has been used to characterize many generating sets of toric ideals previously \cite{sullivant2007toric}. This idea is particularly useful in our setting, since all binomial Parke--Taylor relations for a small number of particles can be computed using computer algebra systems such as \texttt{Macaulay2} as previously discussed. 

\begin{conjecture}
\label{conj:toric-quads-gen-ker}
Let $\cb_n$ be the set of binomials described in Theorem \ref{thm:toricquadratics}. Then for $n\geq 7$ the set $\cb_n$ generates $\ker_\zz(A_n)$. In particular, $I_{A_n} = \langle z^{u^+} - z^{u^-} ~:~ u \in \mathcal{B}_n \rangle : \prod_{\sigma \in \Sigma_n} z_\sigma^\infty$.
\end{conjecture}

Observe that to prove the above conjecture, it suffices to show that $\dim(\mathrm{span}_\zz(\cb_n)) = \dim(\ker_\zz(A_n)) = (n-2)! - \binom{n-1}{2} + 1$ since $\cb_n \subseteq \ker_\zz(A_n)$ by construction. However, the cardinality of the set $\cb_n$ is much larger than $(n-2)! - \binom{n-1}{2} + 1$ and thus its elements will not be a basis for $\ker_\zz(A_n)$. 

\begin{proposition}
Conjecture \ref{conj:toric-quads-gen-ker} holds for $n = 7, 8, 9$. 
\end{proposition}
\begin{proof}
We do this by explicit computation in $\texttt{Macaulay2}$ and the corresponding code can be found in the file \texttt{ParkeTaylorToric.m2} \cite{zenodo}. 
\end{proof}

\begin{remark}
We note that for $n = 5, 6$, the $\ker_\zz(A)$ cannot be generated using only quadratics. In the case of $n = 5$, we saw in Example \ref{example:toric-5} that $\ker_\zz(A_5)$, and thus $I_{A_5}$ are generated by a single cubic relation. In the case of $n = 6$, the set of all quadratic generators of $I_A$ only corresponds to a 14-dimensional subspace of $\ker_\zz(A_6)$ which has dimension $15 = 24 - (\binom{n-1}{2} - 1)$ so unfortunately $\ker_\zz(A_6)$ cannot be generated by only quadratics. 
\end{remark}

We end this section with Table \ref{table:toric-gens} which summarizes the number of generators of both $I_{A_n}$ and $\ker_\zz(A_n)$ for $n \leq 7$. In the next section we will show that the basis of $\mathrm{ker}_{\mathbb{Z}}(A)$, together with \emph{lifts} of the Pl\"ucker relations on $\mathrm{Gr}(2,n)$,
suffices to generate the defining relations of the open Parke--Taylor variety $\mathrm{PT}^\circ_n$ and define $\mathrm{PT}_n$ as a scheme.

\begin{table}
\centering
\begin{tabular}{|c|c|c|c|c|c|c|}
\cline{2-7}
\multicolumn{1}{c|}{} & \backslashbox{n}{Degree} & 2 & 3 & 4 & 5 & 6 \\
\hline
\hline
\multirow{4}{4em}{$I_{A_n}$} & 5 & 0  & 1   & 0 & 0 & 0 \\
& 6 & 24 & 164 & 6 & 0 & 0 \\
& 7 &  1530  & 16410    & 3090  & 372 &  0 \\
\hline
\hline
\multirow{4}{4em}{$\ker_\zz(A_n)$} & 5 & 0  & 1   & 0 & 0 & 0 \\
& 6 &  14    & 1 & 0 & 0 & 0 \\
& 7 &  106   & 0 & 0 & 0 & 0 \\
& 8 &  700   & 0 & 0 & 0 & 0 \\
& 9 &  5013  & 0 & 0 & 0 & 0 \\
\hline
\end{tabular}
\caption{The number of minimal generators for both the ideal $I_{A_n}$ and $\ker_\zz(A_n)$ in each degree for small values of $n$. Note that there are many possible minimal generating sets of $\ker_\zz(A_n)$ which may yield generators of different degrees. This table presents the degrees of the minimal basis described in Theorem \ref{thm:toricquadratics}.}
\label{table:toric-gens}
\end{table}

\section{The ideal of a Parke--Taylor variety} 
\label{sec:idealPT}
With the binomial Parke--Taylor relations understood, we are now ready to compute the whole ideal $\mathcal{I}(\mathrm{PT}_n)$. We start by introducing some notation. Let $T_n$ be the toric variety defined in Section \ref{sec:toric} and write $\mathcal{I}(T_n)$ for its ideal in $\mathbb{C}[z_\sigma\ |\ \sigma\in \Sigma_n]$ and $\tilde{\mathcal{I}}(T_n)$ for the ideal in the Laurent ring $\mathbb{C}[z_\sigma^\pm\ |\ \sigma\in \Sigma_n]$ which is minimally generated by a set of binomials which correspond to a basis for the integer kernel $\ker_\zz(A_n)$.
Now, let $W$ be the affine space defined in the proof of Proposition \ref{prop:toricdim}.
It follows from the proof of Proposition \ref{prop:toricdim} that the open varieties $(\mathbb{C}^*)^{(n-2)!-1}\cap T_n$ and $(\mathbb{C}^*)^{\binom{n-1}{2}-1}\subset W$ are isomorphic via a map which we denote by $\psi$. 
This induces the ring isomorphism $\psi^*:R \to S$, where $R:=\mathbb{C}[p_{ij}^\pm\ | 1\leq i<j<n, (i,j)\neq(1,2)]$ and $S:=\mathbb{C}[z_\sigma^\pm\ |\ \sigma\in S_n]/\tilde{\mathcal{I}}(T_n)$. Finally, let $J$ be the ideal of $R$ generated by all the Pl\"ucker relations for $\mathrm{Gr}(2,n)$ that only involve the variables of $R$.

\begin{theorem} \label{thm:fullideal}
   The ideal of the open Parke--Taylor variety $\Tilde{\mathcal{I}}(\mathrm{PT}^\circ_n)$ in $S$ is $\psi^*(J)$, and thus we have $\mathcal{I}(\mathrm{PT}_n^\circ)=\tilde{\mathcal{I}}(T_n) +\psi^*(J)$. For the closed Parke--Taylor variety we have $\mathcal{I}(\mathrm{PT}_n):=\mathrm{num}(\mathcal{I}(PT_n^\circ)):(\prod z_\sigma)^\infty$, where $\mathrm{num}(\mathcal{I}(PT_n^\circ))$ is the ideal in the polynomial ring $\mathbb{C}[z_\sigma]$ generated by the numerators of all elements of $\mathcal{I}(PT_n^\circ)$. 
\end{theorem}

\begin{proof}
    The statement for $\pt^\circ_n$ follows from the fact that $\psi$ restricts to an isomorphism between $\mathcal{M}_{0,n}$ and $\mathrm{PT}^\circ_n$, which follows from Propositions \ref{thm:PTM0n} and \ref{prop:toricdim}. The statement for the variety $\pt_n$ is then just an application of \cite[Proposition 2.3]{maclagan2010chow}.
\end{proof}

\begin{remark}
Since $T_n$ is parametrized by unsigned products of the Pl\"ucker coordinates, we must slightly alter the binomials of Section \ref{sec:toric} to get binomials which actually vanish on $\pt_n$. In particular, we must apply the map $z_\sigma \mapsto (-1)^{\mathrm{cdes(\sigma)}}z_\sigma$ where $\mathrm{cdes}(\sigma)$ is the so-called cyclic descent number of $\sigma$. Applying this map to any binomial in $\mathcal{I}(T_n)$ yields a binomial which actually vanishes on $\pt_n$. 
\end{remark}

To make this theorem concrete, we need a description of the generators of $\Tilde{\mathcal{I}}(T_n)$ and $\psi^*(J)$. The former was given in Section \ref{sec:toric}, and we now present the latter. For this we will need the following definition.

\begin{definition}
A polynomial $F \in \cc[z_\sigma ~:~ \sigma \in \Sigma_n]$ is called a \emph{lift} of a polynomial $f \in \cc[p_{ij} ~:~ 1 \leq i < j \leq n]$ under $\varphi_n^\ast$ if $\varphi^\ast_n(F) = m f $ and thus $\psi^\ast(mf) = F$. 
\end{definition}

Clearly $\psi^*(J)$ is generated by the set of lifts $\mathcal{F}$ of the Pl\"ucker relations in $J$. 
Before giving an explicit description of these lifts, we first introduce a lemma which is analogous to Lemma \ref{lemma:lift-binomials} and can be used to lift Parke--Taylor relations which hold for small number of particles to those with an arbitrary number. 

\begin{lemma}
\label{lemma:small-plucker-to-big}
Let $f = \sum_{j = 1}^m c_j \prod_{\sigma \in A_j } z_\sigma \in \ci(\mathrm{PT}_n)$ where each $A_i$ is a multiset of permutations in $\Sigma_n$. If there exists an index $i \in \{2,\ldots, n\}$ such that $\{\sigma_i \sigma_{i+1} ~:~ \sigma \in A_j \} = \{\tau_i \tau_{i+1} ~:~ \tau \in A_\ell \}$ as multisets for all $j, \ell \in [m]$, then for any $k \in \zz_{>0}$, it holds that
\[
\sum \limits_{j = 1}^n c_j \prod_{\sigma \in A_j} z_{1 2 \cdots \sigma_i \delta \sigma_{i+1} \cdots \sigma_n} \in \ci(\mathrm{PT_{n+k}}),
\]
where $\delta = \delta_1 \cdots \delta_k$ is a permutation on the letters $n+1, \ldots, n+k$. 
\end{lemma}

The proof of this lemma is quite similar to that of Lemma \ref{lemma:lift-binomials} so we omit it. We now have all of the tools necessary to describe the lifts of Pl\"ucker relations. 

\begin{proposition}
\label{prop:plucker-lifts}
Let $(i,j,k,l) \in [n]$ be a tuple of distinct integers such that $|\{i,j,k,l\} \cap\{1, 2\}| \leq 1$ and $i < j <k < l$ and let $\alpha, \beta$ be permutations whose letters form a partition of $[n] \setminus (\{i,j,k,l\} \cup \{1, 2\})$. Then the following polynomial $F_{ijkl}$ is a lift of the Pl\"ucker relation $p_{ij}p_{kl} - p_{ik}p_{jl} + p_{il}p_{kj}$:
\[
F_{ijkl, \alpha, \beta} = \begin{cases}
z_{12\alpha jlk}z_{12\alpha kjl} + 
z_{12\alpha jkl}z_{12\alpha klj} +
z_{12\alpha jlk}z_{12\alpha klj}, ~~ i = 1 \\
z_{12 klj \alpha}z_{12 ljk \alpha} + 
z_{12 lkj \alpha}z_{12 jlk \alpha} +
z_{12 klj \alpha}z_{12 jlk \alpha}, ~~ i = 2 \\
z_{12\alpha jlki \beta}z_{12\alpha kjli \beta} + 
z_{12\alpha jkli \beta}z_{12\alpha klji \beta} +
z_{12\alpha jlki \beta}z_{12\alpha klji \beta}, ~~ i \geq 3.
\end{cases}
\]
and thus $F_{ijkl, \alpha, \beta} \in \mathcal{I}(\psi^\ast(J))$. 
\end{proposition}
\begin{proof}
We prove the first claim that $F_{ijkl, \alpha, \beta}$ for the case that $i \geq 3$ since the other two cases are similar. Observe that, $F_{ijkl, \emptyset, \emptyset}$ is a polynomial defined for $n = 6$ which satisfies the conditions of Lemma \ref{lemma:small-plucker-to-big} at positions $2$ and $6$. More specifically, at 
positions $2$ and $6$, the multiset of adjacencies in each monomial is $\{2k, 2j\}$ and $\{1i^2\}$ respectively. Thus if $F_{ijkl, \emptyset, \emptyset} \in \psi^\ast(J)$, then $F_{ijkl, \alpha, \beta} \in \psi^\ast(J)$ for any choice of $\alpha, \beta$. Explicit computation then yields that $\tilde{\varphi}_n(F_{ijkl, \emptyset, \emptyset}) = m(p_{ij}p_{kl} - p_{ik}p_{jl} + p_{il}p_{kj})$ as desired. 
\end{proof}

We now carry out this construction explicitly for the case of $n = 5$ and $n = 6$. 

\begin{example}\label{example:fullideal5}
Let $n = 5$ and recall from Theorem \ref{thm:fullideal} that our generating set of $\ci(\mathrm{PT}_n^\circ)$ consists of two parts which are the generators of $\tilde{\ci}(T_n)$ which correspond to a basis for $\ker_\zz(A_n)$ and the lifts of the Pl\"ucker relations which generate $J$. In this case, 
\begin{align*}
\ci(\pt_5^\circ) = \langle f_1, f_2, f_3 \rangle =  \langle &z_{12354}z_{12435}z_{12543}-z_{12345}z_{12453}z_{12534}, \\
&z_{12354}z_{12435}+z_{12345}z_{12453}+z_{12354}z_{12453},\\
&z_{12354}z_{12453}+z_{12453}z_{12534}+z_{12354}z_{12543} \rangle
\end{align*}
where the first generator $f_1$ of the above list comes from $\tilde{\ci(T_n)}$ and is the same polynomial which appears in Example \ref{example:toric-5}. The last two polynomials, $f_2$ and $f_3$, are the lifts of the Pl\"ucker relations $p_{13}p_{45} - p_{14}p_{35} + p_{15}p_{34}$ and $p_{23}p_{45} - p_{24}p_{35} + p_{25}p_{34}$ respectively. As stated in Theorem \ref{thm:fullideal}, we have $\ci(\pt_n) = \ci(\pt_n^\circ) : \prod_{\sigma} z_\sigma ^\infty$ which in this case yields the 5 quadratic generators we saw in Example \ref{ex:LC-to-PT-5}. That is, 
\begin{align*}
\ci(\pt_5) = 
\langle 
z_{12354}z_{12435}+z_{12345}z_{12453}+z_{12354}z_{12453}, z_{12345}z_{12534}+z_{12345}z_{12543}+z_{12354}z_{12543},\\
z_{12354}z_{12435}+z_{12345}z_{12534}+z_{12435}z_{12534}, z_{12345}z_{12453}+z_{12345}z_{12543}+z_{12435}z_{12543},\\
z_{12435}z_{12534}+z_{12453}z_{12534}+z_{12435}z_{12543}
\rangle.
\end{align*}
Lastly, we note that the extra generators arising from saturation essentially belong to $\ci(\pt_n^\circ)$ already, however they are not minimal in $\cc[z_\sigma^\pm]$ but become minimal when passing to the polynomial ring $\cc[z_\sigma]$. For example, the last generator of $\ci(\pt_5)$ above can be written as
\[
z_{12435}z_{12534}+z_{12453}z_{12534}+z_{12435}z_{12543}
= \frac{1}{z_{12354}}f_1 + \frac{z_{12534}}{z_{12354}}f_2 \in \ci(\pt_n^\circ) \subseteq \cc[z_\sigma^\pm]
\]
and thus we see that while $g_1$ is a minimal generator of $\ci(\pt_n) \subset \cc[z_\sigma]$, it is not a minimal generator of $\ci(\pt_n^\circ) \subset \cc[z_\sigma^\pm]$.  
\end{example}

\begin{example} \label{example:fullideal6}
Let $n = 6$ and just as in the previous example, our generating set again consists of the binomials which generate $\tilde{\ci}(T_n)$ and the lifts of the Pl\"ucker relations which generate $J$. The ideal $\tilde{\ci}(T_n)$ is generated by the 15 binomials appearing in Example \ref{example:toric-6}, which we denote by $\cb_6'$, that correspond to a basis for $\ker_\zz(A_6)$. Note that the set of quadratic generators $\cb_6$ given in Theorem \ref{thm:toricquadratics} is contained in this but for $n \leq 7$, these quadratics do not suffice to generate $\ker_\zz(A_n)$ and thus $\cb_6'$ also includes a cubic generator. Also, we must adjust the signs of the generators in Section 4 as mentioned above. For example, applying the transformation $z_\sigma \mapsto (-1)^{\mathrm{cdes(\sigma)}}z_\sigma$ yields
\[
-z_{125463}z_{126534}+z_{125364}z_{126543}  \to -z_{125463}z_{126534}-z_{125364}z_{126543}
\]
the latter of which vanishes on $\pt_n$ while the former does not but instead vanishes on the toric Parke-Taylor variety $T_n$. 

The second part of our generating set of $\ci(\mathrm{PT}_n^\circ)$ is the generators 
$\mathcal{F}$ of $\mathcal{I}(\psi^\ast(J))$ which are described in Proposition \ref{prop:plucker-lifts}. In this case, these polynomials are:

\begin{minipage}{0.3\textwidth}
\begin{align*}
p_{15}p_{34}-p_{14}p_{35}+p_{13}p_{45} \\
p_{25}p_{34}-p_{24}p_{35}+p_{23}p_{45}\\
p_{16}p_{34}-p_{14}p_{36}+p_{13}p_{46}\\
p_{26}p_{34}-p_{24}p_{36}+p_{23}p_{46}\\
p_{16}p_{35}-p_{15}p_{36}+p_{13}p_{56}\\
p_{26}p_{35}-p_{25}p_{36}+p_{23}p_{56}\\
p_{16}p_{45}-p_{15}p_{46}+p_{14}p_{56}\\
p_{26}p_{45}-p_{25}p_{46}+p_{24}p_{56}\\
p_{36}p_{45}-p_{35}p_{46}+p_{34}p_{56}
\end{align*}
\end{minipage}
\begin{minipage}{0.1\textwidth}
\hspace*{6pt}
$\xlongrightarrow[]{\psi^\ast}$
\end{minipage}  
\begin{minipage}{0.4\textwidth}
\begin{align*}
 z_{126354}z_{126435}+z_{126345}z_{126453}+z_{126354}z_{126453}\\
z_{123546}z_{124536}+z_{124536}z_{125346}+z_{123546}z_{125436}\\
z_{125364}z_{125436}+z_{125346}z_{125463}+z_{125364}z_{125463}\\
z_{123645}z_{124635}+z_{124635}z_{126345}+z_{123645}z_{126435}\\
z_{124365}z_{124536}+z_{124356}z_{124563}+z_{124365}z_{124563}\\
z_{123654}z_{125634}+z_{125634}z_{126354}+z_{123654}z_{126534}\\
z_{123465}z_{123546}+z_{123456}z_{123564}+z_{123465}z_{123564}\\
z_{124653}z_{125643}+z_{125643}z_{126453}+z_{124653}z_{126543}\\
z_{124653}z_{125463}+z_{124563}z_{125643}+z_{124653}z_{125643}
\end{align*}
\end{minipage}  
\vspace*{0.5cm}

So the ideal $\ci(\mathrm{PT}_n^\circ)$ is generated by the 9 quadratics above on the right and the 15 binomials which appear in Example \ref{example:toric-6}. The Parke--Taylor ideal $\ci(\mathrm{PT}_n) = \ci(\mathrm{PT}_n^\circ):(\prod z_\sigma)^\infty$ is then obtained by computing the aforementioned saturation. The resulting ideal is minimally generated by 175 quadratic polynomials. 
\end{example}

Note that if Conjecture \ref{conj:toric-quads-gen-ker} holds, then for $n \geq 7$ then the set of quadratic binomials~$\cb_n$ from Theorem \ref{thm:toricquadratics} can be used as the generating set of $\tilde{\ci}(T_n)$. Since the lifts of the Pl\"ucker relation are quadratic by construction, this suggests a set of quadratic generators of $\ci(\pt_n)$.

\begin{conjecture}
The ideal $\tilde{\ci}(\pt_n^\circ) \subseteq \cc[z_\sigma^\pm]$ is quadratically generated by the lifts $\mathcal{F}$ of the Pl\"ucker relations and the quadratic binomials $\cb_n$ from Theorem \ref{thm:toricquadratics}.  
\end{conjecture}

Note that Conjecture \ref{conj:toric-quads-gen-ker} immediately implies this conjecture holds for $n \geq 7$ since we know that the lifts $\mathcal{F}$ together with any generating set for $\tilde{\ci}(T_n)$ suffice to generate $\tilde{\ci}(\pt_n^\circ)$. 

\medskip
\noindent\textbf{Acknowledgements.} The authors would like to thank Bernd Sturmfels for asking the questions that lead to this project and for many helpful comments, and Hadleigh Frost for interesting discussions. We would also like to thank the anonymous referees for their useful suggestions. BH was partially supported by the Alexander von Humboldt Foundation. This research is funded by the European Union (ERC, UNIVERSE PLUS, 101118787). Views and opinions expressed are, however, those of the authors only and do not necessarily reflect those of the European Union or the European Research Council Executive Agency. Neither the European Union nor the granting authority can be held responsible for them. 

\bibliographystyle{plain}
\bibliography{bibliography}
\end{document}